\documentclass[10pt]{amsart}
\usepackage{amsmath,amssymb,latexsym,soul,cite,amsthm,color,enumitem,graphicx,tikz, mathtools,microtype}
\usepackage[colorlinks=true,urlcolor=bluegray,citecolor=bluegray,linkcolor=bluegray,linktocpage,pdfpagelabels,bookmarksnumbered,bookmarksopen]{hyperref}
\definecolor{bluegray}{rgb}{0.4, 0.6, 0.8}
\usepackage[english]{babel}
\usepackage[left=3.2cm,right=3.2cm,top=2.9cm,bottom=2.9cm]{geometry}

\numberwithin{equation}{section}

\newtheorem{theorem}{Theorem}[section]
\theoremstyle{plain}
\newtheorem{lemma}[theorem]{Lemma}
\theoremstyle{plain}
\newtheorem{proposition}[theorem]{Proposition}
\theoremstyle{plain}

\theoremstyle{definition}

\newtheorem{example}[theorem]{Example}

\newcommand{\N}{{\mathbb N}}

\newcommand{\R}{{\mathbb R}}
\newcommand{\eps}{\varepsilon}
\newcommand{\beq}{\begin{equation}}
\newcommand{\eeq}{\end{equation}}
\renewcommand{\le}{\leqslant}
\renewcommand{\ge}{\geqslant}

\newcommand{\restr}[2]{\left.#1\right|_{#2}}
\newcommand{\w}{W^{1,p}(\Omega)}

\newenvironment{enumroman}{\begin{enumerate}

}{\end{enumerate}}

\title[Asymmetric Robin problems]{Constant sign and nodal solutions for nonhomogeneous Robin boundary value problems with asymmetric reactions}

\author[A.\ Iannizzotto, M.\ Marras, N.S.\ Papageorgiou]{Antonio Iannizzotto, Monica Marras, and Nikolaos S.\ Papageorgiou}

\address[A.\ Iannizzotto, M.\ Marras]{Department of Mathematics and Computer Science
\newline\indent
University of Cagliari
\newline\indent
Viale L. Merello 92, 09123 Cagliari, Italy}
\email{antonio.iannizzotto@unica.it, mmarras@unica.it}

\address[N.S.\ Papageorgiou]{Department of Mathematics
\newline\indent
National Technical University
\newline\indent
Zografou Campus, Athens 15780, Greece}
\email{npapg@math.ntua.gr}


\begin{document}

\begin{abstract}
We study a nonlinear, nonhomogeneous elliptic equation with an asymmetric reaction term depending on a positive parameter, coupled with Robin boundary conditions. Under appropriate hypotheses on both the leading differential operator and the reaction, we prove that, if the parameter is small enough, the problem admits at least four nontrivial solutions: two of such solutions are positive, one is negative, and one is sign-changing. Our approach is variational, based on critical point theory, Morse theory, and truncation techniques.
\end{abstract}

\maketitle

\begin{center}
Version of \today\
\end{center}

\section{Introduction}\label{sec1}

\noindent
We study the following nonlinear, nonhomogeneous Robin problem:
\beq\label{rp}
\begin{cases}
-{\rm div}\,a(\nabla u)+\xi(x)|u|^{p-2}u=\lambda g(x,u)+f(x,u) & \text{in $\Omega$} \\
\displaystyle\frac{\partial u}{\partial n_a}+\beta(x)|u|^{p-2}u=0 & \text{on $\partial\Omega$.}
\end{cases}
\eeq
Here $\Omega\subset\R^N$ ($N>1$) is a bounded domain with a $C^2$-boundary $\partial\Omega$, $p>1$, and $a:\R^N\to\R^N$ is a continuous, monotone mapping (hence maximal monotone too) which satisfies certain growth and regularity conditions (see hypotheses ${\bf H}_a$ below). These conditions are mild enough to include in our framework many non-linear operators of interest, such as the $p$-Laplacian, the $(p,q)$-Laplacian, and the generalized mean curvature operator. The potential function $\xi\in L^\infty(\Omega)$ is indefinite (i.e., sign-changing, see hypothesis ${\bf H}_\xi$).
\vskip2pt
\noindent
On the right-hand side, $\lambda>0$ is a parameter and $g,f:\Omega\times\R\to\R$ are Carath\'eodory functions. We assume that for a.a.\ $x\in\Omega$ the mapping $g(x,\cdot)$ is strictly $(p-1)$-sublinear at $\pm\infty$ (concave nonlinearity), while $f(x,\cdot)$ exhibits an asymmetric behavior, being $(p-1)$-superlinear at $+\infty$ and asymptotically $(p-1)$-linear at $-\infty$ (see hypotheses ${\bf H}_g$, ${\bf H}_f$ below). So, in the positive semiaxis we have a competition phenomenon between a concave and a convex nonlinearity, while in the negative semiaxis and in the particular case of the $p$-Laplacian the equation may be resonant with respect to the first eigenvalue.
\vskip2pt
\noindent
In the boundary condition, $\partial u/\partial n_a$ denotes the generalized normal derivative corresponding to the mapping $a$, namely the extension of
\[u\mapsto\langle a(\nabla u),n\rangle, \ u\in C^1(\overline\Omega)\]
to $W^{1,p}(\Omega)$ ($n$ denotes the outward unit normal to $\partial\Omega$). The boundary coefficient $\beta\in C^{0,\alpha}(\partial\Omega)$ is non-negative, and the special case $\beta=0$ corresponds to the Neumann problem (see hypothesis ${\bf H}_\beta$ below).
\vskip2pt
\noindent
In this paper, using variational methods based on the critical point theory, together with suitable truncation/perturbation techniques and Morse theory, we prove that, for $\lambda>0$ small enough, problem \eqref{rp} has at least four nontrivial solutions: two positive, one negative, and one nodal (see Theorem \ref{main} below).
\vskip2pt
\noindent
Recently, elliptic boundary value problems with asymmetric reactions were studied in \cite{MMP,RR,RR1} (semilinear Dirichlet problems with zero potential), \cite{PR} (semilinear Neumann problem with indefinite potential), \cite{DMP,MP1,PR1,PR2} (semilinear Robin problems with indefinite potential). For nonlinear elliptic equations we mention \cite{MP,MMP1} (Dirichlet problems driven by the $p$-Laplacian), \cite{PR3,PZ} (Dirichlet problems driven by the $(p,2)$-Laplacian).
\vskip2pt
\noindent
The paper has the following structure: in Section \ref{sec2} we introduce our hypotheses and main result, and we also establish some preliminary results and notations; in Section \ref{sec3} we deal with constant sign solutions; and in Section \ref{sec4} we investigate extremal constant sign solutions and nodal solutions.

\section{Hypotheses and main result}\label{sec2}

\noindent
We start this section by introducing and commenting the precise hypotheses on all features of problem \eqref{rp}. We begin with the mapping $a$ \text{red}{(please check all hypotheses)}:
\begin{itemize}[leftmargin=1cm]
\item[${\bf H}_a$] $a:\R^N\to\R^N$ is defined by $a(y)=a_0(|y|)y$ for all $y\in\R^N$ with $a_0:\R_+\to\R_+$, and we set for all $t\ge 0$
\[H_0(t)=\int_0^t a_0(\tau)\tau\,d\tau\]
and for all $y\in\R^N$ $H(y)=H_0(|y|)$. Morover:
\begin{enumroman}
\item\label{ha1} $a_0\in C^1(0,+\infty)$, $a_0(t)>0$ for all $t>0$, $t\mapsto a_0(t)t$ is strictly increasing on $(0,+\infty)$, and
\[\lim_{t\to 0^+} a_0(t)t=0, \ \lim_{t\to 0^+}\frac{a_0'(t)t}{a_0(t)}>-1;\]
\item\label{ha2} there exists $\theta\in C^1(0,+\infty)$ s.t.\ for all $t>0$
\[c_0\le\frac{\theta'(t)t}{\theta(t)}\le c_1, \ c_2t^{p-1}\le\theta(t)\le c_3(t^{\sigma-1}+t^{p-1}) \ (c_0,c_1,c_2,c_3>0,\,1\le\sigma<p),\]
and for all $y\in\R^N\setminus\{0\}$
\[|\nabla a(y)|\le c_4\frac{\theta(|y|)}{|y|} \ (c_4>0);\]
\item\label{ha3} for all $y,z\in\R^N$, $y\neq 0$
\[\langle\nabla a(y)z,z\rangle\ge\frac{\theta(|y|)|z|^2}{|y|};\]
\item\label{ha4} there exists $r\in (1,p]$ s.t.\ $t\mapsto H_0(t^\frac{1}{r})$ is convex,
\[\limsup_{t\to 0^+}\frac{r H_0(t)}{t^r}\le c_5 \ (c_5>0),\]
and for all $t\ge 0$
\[pH_0(t)-a_0(t)t^2\ge -c_6 \ (c_6>0).\]
\end{enumroman}
\end{itemize}

\noindent
Hypotheses ${\bf H}_a$ \ref{ha1} - \ref{ha3} are dictated by the nonlinear regularity theory of \cite{L} and the nonlinear maximum principle of \cite{PS}. Hypothesis ${\bf H}_a$ \ref{ha4} serves the needs of our problem but is general enough to include several cases of interest (see Example \ref{exa} below). As a whole, ${\bf H}_a$ implies that $H_0$ is strictly convex and increasing on $\R_+$, and $H$ is convex with $H(0)=0$, $\nabla H(0)=0$, and $\nabla H(y)=a(y)$ for all $y\in\R^N\setminus\{0\}$, i.e., $H$ is the primitive of $a$. This, along with convexity, clearly implies for all $y\in\R^N$
\beq\label{conv}
H(y)\le \langle a(y),y\rangle.
\eeq
Hypotheses ${\bf H}_a$ \ref{ha1} - \ref{ha3} and \eqref{conv} lead to the following properties of $a$ and $H$:

\begin{lemma}\label{ah}
If ${\bf H}_a$ \ref{ha1} - \ref{ha3} hold, then
\begin{enumroman}
\item\label{ah1} $a:\R^N\to\R^N$ is continuous and monotone (hence maximal monotone);
\item\label{ah2} $|a(y)|\le c_7(1+|y|^{p-1})$ for all $y\in\R^N$ ($c_7>0$);
\item\label{ah3} $\displaystyle\langle a(y),y\rangle \ge \frac{c_2|y|^p}{p-1}$ for all $y\in\R^N$;
\item\label{ah4} $\displaystyle\frac{c_2|y|^p}{p(p-1)} \le H(y) \le c_8(1+|y|^p)$ for all $y\in\R^N$ ($c_8>0$).
\end{enumroman}
\end{lemma}

\noindent
In what follows we shall denote $A:\w\to\w^*$ the nonlinear differential operator defined for all $u,v\in\w$ by
\[\langle A(u),v\rangle = \int_\Omega\langle a(\nabla u),\nabla v\rangle\,dx,\]
which is well defined by virtue of ${\bf H}_a$ \ref{ha2}. Such operator enjoys the $(S)_+$-property, i.e., whenever $(u_n)$ is a sequence in $\w$ s.t.\ $u_n\rightharpoonup u$ in $\w$ and
\[\limsup_n\,\langle A(u_n),u_n-u\rangle \le 0,\]
then $u_n\to u$ in $\w$ (see \cite[p.\ 405]{MMP2}). Here follow some examples:

\begin{example}\label{exa}
The following maps $a:\R^N\to\R^N$ satisfy ${\bf H}_a$:
\begin{itemize}[leftmargin=1cm]
\item[$(a)$] $a(y)=|y|^{p-2}y$, corresponding to the $p$-Laplace operator
\[\Delta_p u= {\rm div}\,(|\nabla u|^{p-2}\nabla u);\]
\item[$(b)$] $a(y)=|y|^{p-2}y+|y|^{q-2}y$ ($1<q<p< +\infty$), corresponding to the $(p,q)$-Laplace operator
\[\Delta_p u+\Delta_q u = {\rm div}\,\big(|\nabla u|^{p-2}\nabla u+|\nabla u|^{q-2}\nabla u\big).\]
\end{itemize}
Such operators arise in problems of mathematical physics, see \cite{CI} (reaction-diffusion equations), \cite{D} (elementary particles), \cite{W} (plasma physics). Further:
\begin{itemize}[leftmargin=1cm]
\item[$(c)$] $a(y)=(1+|y|^2)^\frac{p-2}{2}y$, corresponding to the generalized $p$-mean curvature operator
\[{\rm div}\,\big((1+|\nabla u|^2)^\frac{p-2}{2} \nabla u\big);\]
\item[$(d)$] $a(y)=\big(2\ln(1+|y|^p)+(1+|y|^p)^{-1}\big)y$, corresponding to the operator
\[{\rm div}\,\Big(2\ln(1+|\nabla u|^p)\nabla u+\frac{\nabla u}{1+|\nabla u|^p}\Big);\]
\item[$(e)$] $a(y)=|y|^{p-2}y+|y|^{p-2}(1+|y|^p)^{-1}y$, corresponding to the operator
\[\Delta_p u+{\rm div}\,\Big(\frac{|\nabla u|^{p-2}\nabla u}{1+|\nabla u|^p}\Big).\]
\end{itemize}
Such operators arise in problems of nonlinear elasticity \cite{FO} and plasticity.
\end{example}

\noindent
The other ingredients of \eqref{rp} are subject to the following hypotheses:

\begin{itemize}[leftmargin=1cm]
\item[${\bf H}_\xi$] $\xi\in L^\infty(\Omega)$.
\item[${\bf H}_\beta$] $\beta\in C^{0,\alpha}(\partial\Omega)$ for some $\alpha\in(0,1)$, $\beta(x)\ge 0$ for all $x\in\partial\Omega$.
\end{itemize}

\noindent
We note that the potential $\xi$ may change sign, and that for $\beta=0$ we recover the Neumann problem. Finally we introduce our hypotheses on the reactions, starting with $g$:

\begin{itemize}[leftmargin=1cm]
\item[${\bf H}_g$] $g:\Omega\times\R\to\R$ is a Carath\'eodory function, for all $(x,t)\in\Omega\times\R$ we set
\[G(x,t)=\int_0^t g(x,\tau)\,d\tau.\]
Moreover:
\begin{enumroman}
\item\label{hg1} for all $\rho>0$ there exists $a_\rho\in L^\infty(\Omega)_+$ s.t.\ for a.a.\ $x\in\Omega$, all $|t|\le\rho$
\[|g(x,t)|\le a_\rho(x);\]
\item\label{hg2} $\displaystyle\lim_{|t|\to +\infty}\frac{g(x,t)}{|t|^{p-2}t}=0$ uniformly for a.a.\ $x\in\Omega$;
\item\label{hg3} there exists $q\in (1,r)$ s.t.\ for a.a.\ $x\in\Omega$, all $t\in\R$
\[g(x,t)t\ge c_9|t|^q \ (c_9>0);\]
\item\label{hg4} $\displaystyle\limsup_{t\to 0}\frac{g(x,t)}{|t|^{q-2}t}\le c_{10}$ uniformly for a.a. $x\in\Omega$ ($c_{10}>0$);
\item\label{hg5} there exists $\delta_0>0$ s.t.\ for a.a.\ $x\in\Omega$, all $|t|\le\delta_0$
\[g(x,t)t \le qG(x,t).\]
\end{enumroman}
\end{itemize}

\noindent
We set $\xi_0=(p-1)\xi/c_2$, $\beta_0=(p-1)\beta/c_2$ ($c_2>0$ as in ${\bf H}_a$ \ref{ha2}) and we denote by $\hat\lambda_1=\hat\lambda_1(p,\xi_0,\beta_0)>0$ the first eigenvalue of the auxiliary problem
\beq\label{ep}
\begin{cases}
-\Delta_p u+\xi_0(x)|u|^{p-2}u=\lambda |u|^{p-2}u & \text{in $\Omega$} \\
\displaystyle\frac{\partial u}{\partial n_p}+\beta_0(x)|u|^{p-2}u=0 & \text{on $\partial\Omega$.}
\end{cases}
\eeq
where
\[\frac{\partial u}{\partial n_p}=\langle|\nabla u|^{p-2}\nabla u,n\rangle\]
($n$ being as usual the outward unit normal to $\partial\Omega$). Now we consider the asymmetric term $f$:

\begin{itemize}[leftmargin=1cm]
\item[${\bf H}_f$] $f:\Omega\times\R\to\R$ is a Carath\'eodory function, for all $(x,t)\in\Omega\times\R$ we set
\[F(x,t)=\int_0^t f(x,\tau)\,d\tau.\]
Moreover:
\begin{enumroman}
\item\label{hf1} for all $\rho>0$ there exists $b_\rho\in L^\infty(\Omega)_+$ s.t.\ for a.a.\ $x\in\Omega$, all $|t|\le\rho$
\[|f(x,t)|\le b_\rho(x);\]
\item\label{hf2} $\displaystyle\lim_{t\to +\infty}\frac{f(x,t)}{t^{p^*-1}}=0$ uniformly for a.a.\ $x\in\Omega$;
\item\label{hf3} $\displaystyle\lim_{t\to +\infty}\frac{f(x,t)}{t^{p-1}}=+\infty$ uniformly for a.a.\ $x\in\Omega$;
\item\label{hf4} $f(x,t)\le c_{11}(t^{p^*-1}+t^{r-1})-c_{12}t^{p-1}$ for a.a.\ $x\in\Omega$, all $t\ge 0$ ($c_{11},c_{12}>0$);
\item\label{hf5} uniformly for a.a.\ $x\in\Omega$
\[-c_{13} \le \liminf_{t\to -\infty}\frac{f(x,t)}{|t|^{p-2}t} \le \limsup_{t\to -\infty}\frac{f(x,t)}{|t|^{p-2}t} \le \frac{c_2\hat\lambda_1}{p-1} \ (c_{13}>0);\]
\item\label{hf6} $\displaystyle\lim_{t\to 0}\frac{f(x,t)}{|t|^{p-2}t}=0$ uniformly for a.a.\ $x\in\Omega$;
\item\label{hf7} there exists $\delta_1>0$ s.t.\ for a.a.\ $x\in\Omega$, all $|t|\le\delta_1$
\[f(x,t)t \ge 0.\]
\end{enumroman}
\end{itemize}

\noindent
Finally, we set for all $\lambda>0$ and all $(x,t)\in\Omega\times\R$
\[e_\lambda(x,t)=\lambda g(x,t)t+f(x,t)t-p\big[\lambda G(x,t)+F(x,t)\big]\]
and we assume the following:

\begin{itemize}[leftmargin=1cm]
\item[${\bf H}_e$] for all $\lambda>0$
\begin{enumroman}
\item\label{he1} there exists $\eta_\lambda\in L^1(\Omega)_+$ s.t.\ for a.a. $x\in\Omega$, all $0\le t\le t'$
\[e_\lambda(x,t) \le e_\lambda(x,t')+\eta_\lambda(x);\]
\item\label{he2} $\displaystyle\lim_{t\to -\infty} e_\lambda(x,t)=+\infty$ uniformly for a.a.\ $x\in\Omega$.
\end{enumroman}
\end{itemize}

\noindent
We will write ${\bf H}$ to mean all hypotheses ${\bf H}_a$, ${\bf H}_\xi$, ${\bf H}_\beta$, ${\bf H}_g$, ${\bf H}_f$, and ${\bf H}_e$.
\vskip2pt
\noindent
By ${\bf H}_g$ \ref{hg2} $g(x,\cdot)$ is strictly $(p-1)$-sublinear at $\pm\infty$, so it gives a 'concave' contribution to the reaction of \eqref{rp}. Hypotheses ${\bf H}_f$ \ref{hf3}, \ref{hf5} imply that $f(x,\cdot)$ has an asymmetric behavior at $\pm\infty$. More precisely, ${\bf H}_f$ \ref{hf3} means that $f(x,\cdot)$ is strictly $(p-1)$-superlinear at $+\infty$, so on $\R_+$ it represents a 'convex' contribution to the reaction, leading to a competition phenomenon (concave-convex nonlinearities). We point out that the $(p-1)$-superlinearity of $f(x,\cdot)$ is not coupled with the usual Ambrosetti-Rabinowitz $(AR)$ condition. Instead we use the less restrictive quasimonotonicity condition ${\bf H}_e$ \ref{he1}, which includes in our framework $(p-1)$-superlinear reactions with a slower growth at $+\infty$, that fail to satisfy $(AR)$. Note that ${\bf H}_e$ \ref{he1} holds whenever we can find $\rho>0$ s.t.\ the mapping
\[t\mapsto\frac{\lambda g(x,t)+f(x,t)}{t^{p-1}}\]
is nondecreasing in $[\rho,+\infty)$ for a.a.\ $x\in\Omega$ \cite{LY}. On the negative semiaxis $\R_-$, by ${\bf H}_f$ \ref{hf5} the mapping $f(x,\cdot)$ is asymptotically $(p-1)$-linear at $-\infty$, and in the special case of the $p$-Laplacian (Example \ref{exa} $(a)$ with $c_2=p-1$) resonance with the principal eigenvalue is allowed. Resonance occurs from the left, so by ${\bf H}_e$ \ref{he2} problem \eqref{rp} is coercive on the negative direction, which permits the use of the direct method of the calculus of variations. Finally we remark that by ${\bf H}_f$ \ref{hf2} $f(x,\cdot)$ does not satisfy the usual subcritical growth. Instead we have 'almost-critical' growth, namely for all $\eps>0$ we can find $c_\eps>0$ s.t.\ for a.a.\ $x\in\Omega$, all $t\in\R$
\[|f(x,t)|\le \eps|t|^{p^*-1}+c_\eps.\]
This kind of growth is a source of technical difficulties, since $\w$ is not compactly embedded into $L^{p^*}(\Omega)$. We shall overcome such difficulties by using Vitali's theorem.

\begin{example}\label{exgf}
The following functions (of the type $\lambda g+f$, $\lambda>0$) satisfy hypotheses ${\bf H}_g$, ${\bf H}_f$, and ${\bf H}_e$:
\begin{itemize}[leftmargin=1cm]
\item[$(a)$] $\displaystyle t\mapsto\lambda|t|^{q-2}t+\begin{cases}
\hat\lambda_1|t|^{p-2}t & \text{if $t\le 0$} \\
t^{s-1}+t^{r-1} & \text{if $t>0$}
\end{cases} \ (q<r<p<s<p^*);$
\item[$(b)$] $\displaystyle t\mapsto\lambda|t|^{q-2}t+\begin{cases}
\hat\lambda_1|t|^{p-2}t & \text{if $t\le 0$} \\
t^{p-1}\ln(1+t)-t^{r-1} & \text{if $t>0$}
\end{cases} \ (q<r<p);$
\item[$(c)$] $\displaystyle t\mapsto\lambda|t|^{q-2}t+\begin{cases}
\hat\lambda_1|t|^{p-2}t & \text{if $t\le 0$} \\
\displaystyle\frac{t^{p^*-1}}{\ln(1+t^p)}-\frac{p t^{p^*+p-1}}{p^*(1+t^p)\ln(1+t^p)^2}-t^{r-1} & \text{if $t>0$}
\end{cases} \ (q<r<p);$
\end{itemize}
Note that $(a)$ satisfies $(AR)$ while $(b)$ does not, and that $(c)$ has an almost critical growth at $+\infty$.
\end{example}

\noindent
Our main result is the following:

\begin{theorem}\label{main}
If ${\bf H}$ hold, then there exists $\lambda^*>0$ s.t.\ for all $\lambda\in(0,\lambda^*)$ problem \eqref{rp} admits at least four nontrivial solutions $u_+,v_+,u_-,\tilde u\in C^1(\overline\Omega)$ with $u_+$, $v_+$ positive in $\overline\Omega$, $u_-$ negative in $\overline\Omega$, and $\tilde u$ nodal.
\end{theorem}

\subsection{Notation}

\noindent
We establish some notation: we set $\R_+=[0,+\infty)$, $\R_-=(-\infty,0]$; $c_0, c_1, \ldots$ denote positive constants; for all $t\in\R$ we set
\[t^\pm=\max\{0,\pm t\}.\]
In any Banach space $X$, $\rightharpoonup$ denotes weak convergence and $\to$ strong convergence; if $X$ is a function space on the domain $D$, then we denote the positive order cone by
\[X_+=\{u\in X:\,u(x)\ge 0 \ \text{for a.a. $x\in D$}\}.\]
We will say that a functional $\varphi\in C^1(X)$ satisfies the Cerami condition $(C)$, if any sequence $(u_n)$ s.t.\ $(\varphi(u_n))$ is bounded in $\R$ and $(1+\|u_n\|)\varphi'(u_n)\to 0$ in $X^*$, admits a (strongly) convergent subsequence. We will denote the set of critical points of $\varphi$ by
\[K(\varphi)=\{u\in X:\,\varphi'(u)=0\}.\]
We also recall the basic notion from Morse theory: let $\varphi\in C^1(X)$ and $u\in K(\varphi)$ be an isolated critical point, namely there exists a neighborhood $U\subset X$ of $u$ s.t.\ $K(\varphi)\cap U=\{u\}$, and $\varphi(u)=c$. Then, for all $k\in\N$, the $k$-th critical group of $\varphi$ at $u$ is defined by
\[C_k(\varphi,u)=H_k\big(\{v\in U:\,\varphi(v)\le c\},\,\{v\in U:\,\varphi(v)\le c,\,v\neq u\}\big),\]
where $H_k(\cdot,\cdot)$ denotes the $k$-th singular homology group of a topological pair.
\vskip2pt
\noindent
We shall use the function spaces $(\w,\|\cdot\|)$ and $(C^1(\overline\Omega),\|\cdot\|_{C^1(\overline\Omega)})$, endowed with the usual norms. Brackets $\langle\cdot,\cdot\rangle$ denote both the inner product of $\R^N$ and the duality between $\w^*$ and $\w$, with no possible confusing arising. We shall also use the Lebesgue spaces $(L^\nu(\Omega),\|\cdot\|_\nu)$ for all $\nu\in[1,+\infty]$, and the trace space $(L^p(\partial\Omega),\|\cdot\|_{L^p(\partial\Omega)})$ (any $u\in\w$ will be identified with its trace on $\partial\Omega$). We set
\[D_+=\{u\in C^1(\overline\Omega):\,u(x)>0\,\text{for all $x\in\overline\Omega$}\},\]
noting that $D_+\subseteq{\rm int}\,(C^1(\overline\Omega)_+)$.

\section{Constant sign solutions}\label{sec3}

\noindent
For all $\lambda>0$, $u\in\w$ we set
\[\varphi_\lambda(u) = \int_\Omega H(\nabla u)\,dx+\frac{1}{p}\int_\Omega\xi(x)|u|^p\,dx+\frac{1}{p}\int_{\partial\Omega}\beta(x)|u|^p\,d\sigma-\int_\Omega\big[\lambda G(x,u)+F(x,u)\big]\,dx\]
(the integral on $\partial\Omega$ is computed with respect to the $(N-1)$-dimensional Hasudorff measure). By Lemma \ref{ah} \ref{ah4}, ${\bf H}_\xi$, ${\bf H}_\beta$, ${\bf H}_g$ \ref{hg1} \ref{hg2}, ${\bf H}_f$ \ref{hf1} \ref{hf2} \ref{hf5}, we have $\varphi_\lambda\in C^1(\w)$. Moreover, $\varphi_\lambda$ is the energy functional for problem \eqref{rp}. Indeed, for all $u\in K(\varphi_\lambda)$, $v\in\w$ we have
\beq\label{wrp}
\langle A(u),v\rangle+\int_\Omega\xi(x)|u|^{p-2}uv\,dx+\int_{\partial\Omega}\beta(x)|u|^{p-2}uv\,d\sigma = \int_\Omega[\lambda g(x,u)+f(x,u)]v\,dx,
\eeq
i.e., $u$ is a (weak) solution of \eqref{rp}. Besides, let
\beq\label{mu}
\mu>\max\Big\{1,\frac{p-1}{c_2}\Big\}\|\xi\|_\infty
\eeq
and for all $(x,t)\in\Omega\times\R$ set
\[k_\lambda(x,t)=\lambda g(x,t)+f(x,t)+\mu|t|^{p-2}t,\]
\[k^\pm_\lambda(x,t)=k_\lambda(x,\pm t^\pm),\]
and the primitives
\[K^{(\pm)}_\lambda(x,t)=\int_0^t k^{(\pm)}_\lambda(x,\tau)\,d\tau.\]
Now we define two truncated/perturbed functionals by setting for all $u\in\w$
\[\hat\varphi^\pm_\lambda(u) = \int_\Omega H(\nabla u)\,dx+\frac{1}{p}\int_\Omega(\xi(x)+\mu)|u|^p\,dx+\frac{1}{p}\int_{\partial\Omega}\beta(x)|u|^p\,d\sigma-\int_\Omega K^\pm_\lambda(x,u)\,dx\]
(note that $\xi+\mu$ is positive by \eqref{mu}). We shall study separately the properties of $\hat\varphi^+_\lambda$ and $\hat\varphi^-_\lambda$, which are different by the asymmetry of $f$.

\begin{lemma}\label{cp}
If ${\bf H}$ hold, then for all $\lambda>0$ $\hat\varphi^+_\lambda\in C^1(\w)$ satisfies $(C)$.
\end{lemma}
\begin{proof}
Clearly $k^+_\lambda:\Omega\times\R\to\R$ is a Carath\'eodory function, with a growth defined by ${\bf H}_g$ \ref{hg2}, ${\bf H}_f$ \ref{hf2} \ref{hf5}, so $\hat\varphi^+_\lambda\in C^1(\w)$.
\vskip2pt
\noindent
Let $(u_n)$ be a sequence in $\w$ s.t.\ $|\hat\varphi^+_\lambda(u_n)|\le c_{14}$ for all $n\in\N$ ($c_{14}>0$) and $(1+\|u_n\|)(\hat\varphi^+_\lambda)'(u_n)\to 0$ in $\w^*$. We can find a sequence $(\eps_n)$ in $\R$ s.t.\ $\eps_n\to 0^+$ and for all $n\in\N$, $v\in\w$
\begin{align}\label{cp1}
&\Big|\langle A(u_n),v\rangle+\int_\Omega(\xi(x)+\mu)|u_n|^{p-2}u_n v\,dx+\int_{\partial\Omega}\beta(x)|u_n|^{p-2}u_n v\,d\sigma-\int_\Omega k^+_\lambda(x,u_n)v\,dx\Big| \\
\nonumber&\le\frac{\eps_n\|v\|}{1+\|u_n\|}.
\end{align}
Choosing $v=-u_n^-$ in \eqref{cp1} and using Lemma \ref{ah} \ref{ah3} we get for all $n\in\N$
\[\frac{c_2}{p-1}\|\nabla u_n^-\|_p^p+\int_\Omega(\xi(x)+\mu)(u_n^-)^p\,dx \le \eps_n.\]
Passing to the limit we see that $u_n^-\to 0$ in $\w$. Now we deal with $u_n^+$. By definition of $K^+_\lambda$ we have for all $n\in\N$
\begin{align*}
&pc_{14} \ge p\int_\Omega H(\nabla u_n)\,dx+\int_\Omega(\xi(x)+\mu)|u_n|^p\,dx+\int_{\partial\Omega}\beta(x)|u_n|^p\,d\sigma-p\int_\Omega K^+_\lambda(x,u_n)\,dx \\
&\ge p\int_\Omega H(\nabla u^+_n)\,dx+\int_\Omega\xi(x)(u_n^+)^p\,dx+\int_{\partial\Omega}\beta(x)(u_n^+)^p\,d\sigma-p\int_\Omega\big[\lambda G(x,u^+_n)+F(x,u^+_n)\big]\,dx,
\end{align*}
while \eqref{cp1} with $v=-u_n^+$ yields
\[-\langle A(u_n^+),u_n^+\rangle-\int_\Omega\xi(x)(u_n^+)^p\,dx-\int_{\partial\Omega}\beta(x)(u_n^+)^p\,d\sigma+\int_\Omega[\lambda g(x,u_n^+)+f(x,u_n^+)]u_n^+\,dx\le\eps_n.\]
Adding up we get
\[\int_\Omega\big[pH(\nabla u_n^+)-\langle a(\nabla u_n^+),\nabla u_n^+\rangle\big]\,dx+\int_\Omega e_\lambda(x,u_n^+)\,dx\le c_{15} \ (c_{15}>0),\]
which by ${\bf H}_a$ \ref{ha4} implies
\beq\label{cp3}
\int_\Omega e_\lambda(x,u_n^+)\,dx\le c_{16} \ (c_{16}>0).
\eeq
We claim that $(u_n^+)$ is bounded in $\w$. Arguing by contradiction, we may assume that (passing if necessary to a subsequence) $\|u_n^+\|\to +\infty$. Then we set for all $n\in\N$ $w_n=u_n^+\|u_n^+\|^{-1}$, so $w_n\in\w$ with $\|w_n\|=1$. Passing to a subsequence we have $w_n\rightharpoonup w$ in $\w$ and $w_n\to w$ both in $L^p(\Omega)$ and in $L^p(\partial\Omega)$ (due to the compact embeddings $\w\hookrightarrow L^p(\Omega),\,L^p(\partial\Omega)$). Clearly $w\in\w_+$. Two cases may occur:
\begin{itemize}[leftmargin=1cm]
\item[$(a)$] First we assume $w\neq 0$. Let
\[\Omega_+=\{x\in\Omega:\,w(x)>0\},\]
then $|\Omega_+|>0$ and for a.a.\ $x\in\Omega_+$ we have $u_n^+(x)\to +\infty$. By ${\bf H}_f$ \ref{hf3} we have for a.a.\ $x\in\Omega_+$
\[\frac{F(x,u_n^+(x))}{\|u_n^+\|^p} = \frac{F(x,u_n^+(x))}{u_n^+(x)^p}w_n(x)^p\to +\infty\]
By Fatou's lemma we have
\[\lim_n\int_{\Omega_+}\frac{F(x,u_n^+)}{\|u_n^+\|^p}\,dx= +\infty.\]
By ${\bf H}_f$ \ref{hf1} \ref{hf3} we have for a.a.\ $x\in\Omega$, all $t\ge 0$
\[F(x,t)\ge t^p-c_{17} \ (c_{17}>0),\]
so we have
\[\int_{\Omega\setminus\Omega_+}\frac{F(x,u_n^+)}{\|u_n^+\|^p}\,dx \ge \int_{\Omega\setminus\Omega_+}w_n^p\,dx-\frac{c_{17}|\Omega|}{\|u_n^+\|^p},\]
and the latter is bounded from below. Summarizing,
\beq\label{cp4}
\lim_n\int_{\Omega}\frac{F(x,u_n^+)}{\|u_n^+\|^p}\,dx= +\infty.
\eeq
Besides, ${\bf H}_g$ \ref{hg2} implies, as above,
\[\lim_n\int_{\Omega_+}\frac{G(x,u_n^+)}{\|u_n^+\|^p}\,dx= 0.\]
By ${\bf H}_g$ \ref{hg1} \ref{hg2}, for any $\eps>0$ we can find $c_{18}=c_{18}(\eps)>0$ s.t.\ for a.a.\ $x\in\Omega$, all $t\ge 0$
\[G(x,t)\le\frac{\eps}{p}t^p+c_{18}.\]
So we have
\[\limsup_n\int_{\Omega\setminus\Omega_+}\frac{G(x,u_n^+)}{\|u_n^+\|^p}\,dx \le \limsup_n\int_{\Omega\setminus\Omega_+}\Big(\frac{\eps}{p}w_n^p+\frac{c_{18}}{\|u_n^+\|^p}\Big)\,dx \le \frac{\eps}{p}\|w\|_p^p.\]
since $\eps>0$ is arbitrary, adding the two integrals we get
\beq\label{cp5}
\lim_n\int_{\Omega}\frac{G(x,u_n^+)}{\|u_n^+\|^p}\,dx=0.
\eeq
Now \eqref{cp4}, \eqref{cp5} imply
\[\lim_n\int_{\Omega}\frac{\lambda G(x,u_n^+)+F(x,u_n^+)}{\|u_n^+\|^p}\,dx= +\infty.\]
But again from boundedness of $(\hat\varphi^+_\lambda(u_n))$, and recalling that $u^-_n\to 0$ in $\w$, we have for all $n\in\N$
\[\int_\Omega H(\nabla u^+_n)\,dx+\frac{1}{p}\int_\Omega\xi(x)(u_n^+)^p\,dx+\frac{1}{p}\int_{\partial\Omega}\beta(x)(u^+_n)^p\,d\sigma-\int_\Omega\big[\lambda G(x,u^+_n)+F(x,u_n^+)\big]\,dx \ge -c_{19} \ (c_{19}>0),\]
which, along with Lemma \ref{ah} \ref{ah4}, implies
\begin{align*}
&\int_{\Omega}\frac{\lambda G(x,u_n^+)+F(x,u_n^+)}{\|u_n^+\|^p}\,dx \le \frac{c_{19}}{\|u^+_n\|^p}+\int_\Omega\frac{c_8(1+|\nabla u^+_n|^p)}{\|u^+_n\|^p}\,dx+\frac{1}{p}\int_\Omega\xi(x)w_n^p\,dx+\frac{1}{p}\int_{\partial\Omega}\beta(x)w_n^p\,d\sigma \\
&\le c_{20}(1+\|w_n\|^p) \ (c_{20}>0),
\end{align*}
and the latter is bounded from above. Thus we reach a contradiction.
\item[$(b)$] Now we assume $w=0$. Fix $M>0$ and set $\hat w_n=(Mp)^\frac{1}{p}w_n$ for all $n\in\N$, so $\hat w_n\rightharpoonup 0$ in $\w$ and $\hat w_n\to 0$ in $L^p(\Omega)$ and $L^p(\partial\Omega)$. By ${\bf H}_g$ \ref{hg1} \ref{hg2} we have for a.a.\ $x\in\Omega$, all $t\in\R$
\[|G(x,t)|\le c_{21}(1+|t|^p) \ (c_{21}>0).\]
So we get
\beq\label{cp6}
\lim_n\int_\Omega G(x,\hat w_n)\,dx=0.
\eeq
Clearly $(\hat w_n)$ is bounded in $L^{p^*}(\Omega)$, so set
\[K_0=\sup_{n\in\N}\|\hat w_n\|_{p^*}^{p^*}.\]
By ${\bf H}_f$ \ref{hf1} \ref{hf2}, for any $\eps>0$ we can find $c_{21}=c_{22}(\eps)>0$ s.t.\ for a.a.\ $x\in\Omega$, all $t\in\R$
\[|F(x,t)|\le\frac{\eps}{2K_0}|t|^{p^*}+c_{22}.\]
So, the sequence $(F(\cdot,\hat w_n))$ is bounded in $L^1(\Omega)$. Furthermore, for any measurable set $B\subset\Omega$ with $|B|\le\eps(2c_{22})^{-1}$ we have for all $n\in\N$
\[\int_B|F(x,\hat w_n)|\,dx \le \frac{\eps}{2K_0}\|\hat w_n\|_{p^*}^{p^*}+c_{22}|B| \le \eps.\]
So the sequence $(F(\cdot,\hat w_n))$ is uniformly integrable in $\Omega$ (see \cite[Problem 1.6]{GP1}). Passing to a subsequence, we have $F(x,\hat w_n(x))\to 0$ as $n\to\infty$, for a.a.\ $x\in\Omega$. By Vitali's theorem \cite[p.\ 5]{GP1} we have
\beq\label{cp7}
\lim_n\int_\Omega F(x,\hat w_n)\,dx=0.
\eeq
Since $\|u_n^+\|\to +\infty$, for $n\in\N$ big enough we have
\beq\label{cp8}
0<\frac{(Mp)^\frac{1}{p}}{\|u_n^+\|}\le 1.
\eeq
Let $\hat\psi_\lambda^+\in C^1(\w)$ be defined for all $u\in\w$ by
\[\hat\psi_\lambda^+(u)= \frac{c_2}{p(p-1)}\|\nabla u\|_p^p+\frac{1}{p}\int_\Omega(\xi(x)+\mu)|u|^p\,dx+\frac{1}{p}\int_{\partial\Omega}\beta(x)|u|^p\,d\sigma-\int_\Omega K^+_\lambda(x,u)\,dx.\]
For all $n\in\N$ there exists $t_n\in[0,1]$ s.t.
\[\hat\psi^+_\lambda(t_nu_n^+)=\max_{t\in[0,1]}\hat\psi^+_\lambda(tu_n^+).\]
In particular, by \eqref{cp8} we have for $n\in\N$ big enough
\begin{align*}
&\hat\psi^+_\lambda(t_nu_n^+) \ge \hat\psi^+_\lambda(\hat w_n) \\
&\ge \frac{c_2M}{p-1}\|\nabla w_n\|_p^p+M\int_\Omega(\xi(x)+\mu)w_n^p\,dx-\int_\Omega\big[\lambda G(x,\hat w_n)+F(x,\hat w_n)+M\mu w_n^p\big]\,dx \\
&\ge M(c_{23}-\mu\|w_n\|_p^p)-\int_\Omega\big[\lambda G(x,\hat w_n)+F(x,\hat w_n)\big]\,dx \ (c_{23}>0)
\end{align*}
(recall that $\mu>\|\xi	\|_\infty$ and $\|w_n\|=1$). Now by \eqref{cp6}, \eqref{cp7} we have for $n\in\N$ even bigger
\[\hat\psi^+_\lambda(t_nu_n^+) \ge Mc_{24} \ (c_{24}>0)\]
which by arbitrarity of $M>0$ implies $\hat\psi^+_\lambda(t_nu_n^+)\to +\infty$ as $n\to\infty$. By Lemma \ref{ah} \ref{ah4} we have $\hat\varphi^+_\lambda(u)\ge\hat\psi^+_\lambda(u)$ for all $u\in\w$, hence the sequence $(\hat\psi^+_\lambda(u_n^+))$ is bounded from above. Besides, clearly $\hat\psi^+_\lambda(0)=0$. So, for all $n\in\N$ big enough we must have $t_n\in(0,1)$. By definition of $t_n$, then,
\[\frac{d}{dt}\restr{\hat\psi^+_\lambda(tu^+_n)}{t=t_n} = \langle(\hat\psi^+_\lambda)'(t_nu_n^+),u_n^+\rangle=0.\]
Multiplying by $t_n$ we get
\[\frac{c_2}{p-1}\|\nabla(t_nu_n^+)\|_p^p+\int_\Omega\xi(x)(t_nu_n^+)^p\,dx+\int_{\partial\Omega}\beta(x)(t_nu_n^+)^p\,d\sigma=\int_\Omega[\lambda g(x,t_nu_n^+)+f(x,t_nu_n^+)]t_nu_n^+\,dx.\]
By ${\bf H}_e$ \ref{he1}, $t_n<1$, and \eqref{cp3} we have
\[\int_\Omega e_\lambda(x,t_nu_n^+)\,dx \le \int_\Omega e_\lambda(x,u_n^+)\,dx+\|\eta_\lambda\|_1<c_{25} \ (c_{25}>0),\]
which implies
\[\int_\Omega[\lambda g(x,t_nu_n^+)+f(x,t_nu_n^+)]t_nu_n^+\,dx \le p\int_\Omega\big[\lambda G(x,t_nu_n^+)+F(x,t_nu_n^+)\big]\,dx+c_{25}.\]
Thus, for all $n\in\N$ big enough we have
\[p\hat\psi^+_\lambda(t_nu^+_n)=\int_\Omega[\lambda g(x,t_nu_n^+)+f(x,t_nu_n^+)]t_nu_n^+\,dx - p\int_\Omega\big[\lambda G(x,t_nu_n^+)+F(x,t_nu_n^+)\big]\,dx\le c_{25},\]
a contradiction.
\end{itemize}
By the claim above and $u^-_n\to 0$, we see that $(u_n)$ is bounded in $\w$. Passing to a subsequence, we may assume $u_n\rightharpoonup u$ in $\w$ and $u_n\to u$ in $L^p(\Omega)$ and $L^p(\partial\Omega)$. By ${\bf H}_g$ \ref{hg1} \ref{hg2} we have for a.a.\ $x\in\Omega$, all $t\in\R$
\[|g(x,t)|\le c_{26}(1+|t|^{p-1}) \ (c_{26}>0),\]
hence by H\"older's inequality
\beq\label{cp9}
\lim_n\int_\Omega g(x,u_n^+)(u_n-u)\,dx = 0.
\eeq
Besides, $(u_n)$ is bounded in $L^{p^*}(\Omega)$, so we set
\[K_1=\sup_{n\in\N}\|u_n\|_{p^*}+\|u\|_{p^*}.\]
By ${\bf H}_f$ \ref{hf1} \ref{hf2}, for any $\eps>0$ we can find $c_{27}=c_{27}(\eps)>0$ s.t.\ for a.a.\ $x\in\Omega$, all $t\ge 0$
\[|f(x,t)|\le\frac{\eps}{3K_1^{p^*}}t^{p^*-1}+c_{27}.\]
Passing if necessary to a subsequence, we have $f(x,u_n^+(x))(u_n(x)-u(x))\to 0$ as $n\to\infty$, for a.a.\ $x\in\Omega$. Moreover, for any measurable $B\subset\Omega$ with
\[|B|\le\Big(\frac{\eps}{6K_1c_{27}}\Big)^{(p^*)'}\]
we have by H\"older's inequality
\begin{align*}
&\Big|\int_B f(x,u_n^+)(u_n-u)\,dx\Big| \le \frac{\eps}{3K_1^{p^*}}\int_B(u_n^+)^{p^*-1}|u_n-u|\,dx+c_{27}\|u_n-u\|_1 \\
&\le \frac{\eps}{3K_1^{p^*}}\|u_n^+\|_{p^*}^{p^*-1}\|u_n-u\|_{p^*}+c_{27}|B|^\frac{1}{(p^*)'}\|u_n-u\|_{p^*} \\
&\le \frac{2\eps}{3}+\frac{\eps}{3} = \eps.
\end{align*}
So, the sequence  $(f(\cdot,u_n^+)(u_n-u))$ is uniformly integrable in $\Omega$. By Vitali's theorem we get
\beq\label{cp10}
\lim_n\int_\Omega f(x,u_n^+)(u_n-u)\,dx = 0.
\eeq
If we choose $v=u_n-u$ in \eqref{cp1}, pass to the limit as $n\to\infty$, and use \eqref{cp9} and \eqref{cp10}, we now get
\[\limsup_n\langle A(u_n),u_n-u\rangle =0.\]
By the $(S)_+$-property of $A$ we have $u_n\to u$ in $\w$, which concludes the proof.
\end{proof}

\noindent
The following lemmas show that for $\lambda>0$ small enough $\hat\varphi^+_\lambda$ exhibits the 'mountain pass' geometry:

\begin{lemma}\label{mpp}
If ${\bf H}$ hold, then there exists $\lambda^*>0$ s.t.\ for all $\lambda\in(0,\lambda^*)$ there exists $\rho_\lambda>0$ s.t.
\[\inf_{\|u\|=\rho_\lambda}\hat\varphi^+_\lambda(u)=\hat m^+_\lambda >0.\]
\end{lemma}
\begin{proof}
By ${\bf H}_g$ \ref{hg2} \ref{hg4}, for all $\eps>0$ we can find $c_{28}=c_{28}(\eps)>0$ s.t.\ for a.a.\ $x\in\Omega$, all $t\ge 0$
\[G(x,t)\le \frac{\eps}{p}t^p+c_{28}t^q\]
(recall that $q<p$). By ${\bf H}_f$ \ref{hf4} we have as well for a.a.\ $x\in\Omega$, all $t\ge 0$
\[F(x,t)\le \frac{c_{11}}{p^*}t^{p^*}+\frac{c_{11}}{r}t^r-\frac{c_{12}}{p}t^p.\]
Recalling that $q<r<p<p^*$ and choosing $\eps<c_{12}/\lambda$ we get for a.a.\ $x\in\Omega$, all $t\ge 0$
\[\lambda G(x,t)+F(x,t)\le\lambda c_{28}t^q+c_{29}t^{p^*}-\frac{c_{30}}{p}t^p,\]
where, taking $c_{28},c_{29}>0$ big enough, we may assume $c_{30}>\|\xi\|_\infty$. So, recalling ${\bf H}_\beta$ and $\mu>\|\xi\|_\infty$, for all $u\in\w$ we have
\begin{align*}
&\hat\varphi^+_\lambda(u) \ge \frac{c_2}{p(p-1)}\|\nabla u^-\|_p^p+\frac{1}{p}\int_\Omega(\xi(x)+\mu)(u^-)^p\,dx \\
&+\frac{c_2}{p(p-1)}\|\nabla u^+\|_p^p+\frac{1}{p}\int_\Omega\xi(x)(u^+)^p\,dx-\int_\Omega\big[\lambda G(x,u^+)+F(x,u^+)\big]\,dx \\
&\ge c_{31}\|u^-\|^p+\frac{c_2}{p(p-1)}\|\nabla u^+\|_p^p+\frac{1}{p}\int_\Omega(\xi(x)+c_{30})(u^+)^p\,dx-\lambda c_{28}\|u^+\|_q^q-c_{29}\|u^+\|_{p^*}^{p^*} \\
&\ge c_{31}\|u^-\|^p+c_{32}\|u^+\|^p-\lambda c_{33}\|u\|^q-c_{34}\|u\|^{p^*} \\
&\ge c_{35}\|u\|^p-\lambda c_{33}\|u\|^q-c_{34}\|u\|^{p^*} \ (c_{31},\ldots c_{35}>0).
\end{align*}
Summarizing, we have
\beq\label{mpp1}
\hat\varphi^+_\lambda(u) \ge (c_{35}-j_\lambda(\|u\|))\|u\|^p,
\eeq
where we have set for all $t
\rho>0$
\[j_\lambda(\rho)=\lambda c_{33}\rho^{q-p}+c_{34}\rho^{p^*-p}.\]
Since $q<p<p^*$, we have for all $\lambda>0$
\[\lim_{\rho\to 0^+}j_\lambda(\rho)=\lim_{\rho\to +\infty} j_\lambda(\rho)=+\infty,\]
so there exists $\rho_\lambda>0$ s.t.
\[j_\lambda(\rho_\lambda)=\inf_{\rho>0}j_\lambda(\rho).\]
In particular we have
\[0 = j'_\lambda(\rho_\lambda) = \lambda c_{33}(q-p)\rho_\lambda^{q-p-1}+c_{34}(p^*-p)\rho_\lambda^{p^*-p-1},\]
which yields
\[\rho_\lambda=\Big(\frac{\lambda c_{33}(p-q)}{c_{34}(p^*-p)}\Big)^\frac{1}{p^*-q}.\]
We are interested in the mapping $\lambda\mapsto j_\lambda(\rho_\lambda)$, which amounts to
\[j_\lambda(\rho_\lambda) = c_{36}\lambda^\frac{p^*-p}{p^*-q} \ (c_{36}>0 \ \text{independent of $\lambda$}),\]
and the latter tends to $0$ as $\lambda\to 0^+$. So there exists $\lambda^*>0$ s.t.\ for all $\lambda\in(0,\lambda^*)$ we have $j_\lambda(\rho_\lambda)<c_{35}$. Thus, by \eqref{mpp1} we have for all $u\in\w$ with $\|u\|=\rho_\lambda$
\[\hat\varphi^+_\lambda(u) \ge (c_{35}-j_\lambda(\rho_\lambda))\rho_\lambda^p =: \hat m^+_\lambda>0,\]
which concludes the proof.
\end{proof}

\noindent
Let $\hat u_1=\hat u_1(p,\xi_0,\beta_0)\in D_+$ be the positive, $L^p(\Omega)$-normalized first eigenfunction of the eigenvalue problem \eqref{ep}:

\begin{lemma}\label{bp}
If ${\bf H}$ hold, then for all $\lambda>0$
\[\lim_{t\to +\infty}\hat\varphi^+_\lambda(t\hat u_1)=-\infty.\]
\end{lemma}
\begin{proof}
Fix $\lambda>0$. By ${\bf H}_g$ \ref{hg2}, ${\bf H}_f$ \ref{hf3}, for all $\eta>0$ we can find $c_{37}=c_{37}(\eta)>0$ s.t.\ for a.a.\ $x\in\Omega$, all $t\ge c_{37}$
\[\lambda G(x,t)+F(x,t) \ge\frac{\eta-\mu}{p}t^p.\]
Since $\hat u_1\in D_+$, for all $t>0$ big enough we have $t\hat u_1(x)\ge c_{37}$ for all $x\in\overline\Omega$. Then by Lemma \ref{ah} \ref{ah4} we have
\begin{align*}
&\hat\varphi^+_\lambda(t\hat u_1) \le c_8\int_\Omega(1+t^p|\nabla \hat u_1|^p)\,dx+\frac{1}{p}\int_\Omega(\xi(x)+\mu)(t\hat u_1)^p\,dx+\frac{1}{p}\int_{\partial\Omega}\beta(x)(t\hat u_1)^p\,d\sigma-\int_\Omega \frac{\eta}{p}(t\hat u_1)^p\,dx \\
&\le c_{38}+\Big(c_{39}\|\hat u_1\|^p-\frac{\eta}{p}\|\hat u_1\|_p^p\Big)t^p \ (c_{38},c_{39}>0).
\end{align*}
Choosing $\eta>0$ big enough, the latter tends to $-\infty$ as soon as $t\to +\infty$, concluding the proof.
\end{proof}

\noindent
The above lemmas lead, through the use of the mountain pass theorem and constrained minimization, to the existence of two positive solutions:

\begin{proposition}\label{sp}
If ${\bf H}$ hold, then there exists $\lambda^*>0$ s.t.\ for all $\lambda\in(0,\lambda^*)$ problem \eqref{rp} admits at least two positive solutions $u_+,v_+\in D_+$.
\end{proposition}
\begin{proof}
Fix $\lambda\in(0,\lambda^*)$. By Lemmas \ref{mpp}, \ref{bp} we can find $t>0$ s.t.\ $\hat\varphi^+_\lambda(t\hat u_1)<\hat m^+_\lambda$. Recalling also Lemma \ref{cp}, we can apply the mountain pass theorem and find $u_+\in K(\hat\varphi^+_\lambda)$ s.t.\ $\hat\varphi^+_\lambda(u_+)\ge\hat m^+_\lambda>0$. In particular we have $u_+\neq 0$. For all $v\in\w$ we have
\beq\label{sp1}
\langle A(u_+),v\rangle+\int_\Omega(\xi(x)+\mu)|u_+|^{p-2}u_+v\,dx+\int_{\partial\Omega}\beta(x)|u_+|^{p-2}u_+v\,d\sigma=\int_\Omega k^+_\lambda(x,u_+)v\,dx.
\eeq
Choosing $v=-u_+^-$ in \eqref{sp1} and applying Lemma \ref{ah} \ref{ah3} yields
\[\frac{c_2}{p-1}\|\nabla u_+^-\|_p^p+\int_\Omega(\xi(x)+\mu)(u_+^-)^p\,dx \le 0,\]
hence $u_+\in\w_+\setminus\{0\}$. So \eqref{sp1} becomes for all $v\in\w$
\[\langle A(u_+),v\rangle+\int_\Omega\xi(x)u_+^{p-1}v\,dx+\int_{\partial\Omega}\beta(x)u_+^{p-1}v\,d\sigma=\int_\Omega [\lambda g(x,u_+)+f(x,u_+)]v\,dx,\]
i.e., $u_+$ is a solution of \eqref{rp}. Reasoning as in \cite{PR4,PR5} we have $u_+\in L^\infty(\Omega)$. By the nonlinear regularity theory of \cite{L} we have $u_+\in C^1(\overline\Omega)_+\setminus\{0\}$. For a.a.\ $x\in\Omega$ we have by ${\bf H}_g$ \ref{hg3}
\[g(x,u_+(x))\ge 0,\]
while by ${\bf H}_f$ \ref{hf2} \ref{hf6} we have 
\[f(x,u_+(x))\ge -c_{40}u_+(x)^{p^*-1} \ (c_{40}>0).\]
So we have in $\Omega$ (in a weak sense)
\[{\rm div}\,a(\nabla u_+)\le\|\xi\|_\infty u_+^{p-1}+c_{40}u_+^{p^*-1}\le c_{41}u_+^{p-1} \ (c_{41}>0),\]
since $u_+$ is bounded. By the nonlinear maximum principle \cite[pp.\ 111, 120]{PS} we have $u_+\in D_+$.
\vskip2pt
\noindent
Now let $\delta_0>0$ be as in ${\bf H}_g$ \ref{hg5}. By ${\bf H}_a$ \ref{ha4} we can find $\delta_2\in (0,\delta_0)$, $c_{42}>c_5$ s.t.\ for all $y\in\R^N$, $|y|\le\delta_2$
\[H(y)\le c_{42}|y|^r.\]
Fix any $u\in D_+$. For $t>0$ small enough we have $\|tu\|_{C^1(\overline\Omega)}\le\delta_2$, in particular $|\nabla (tu)(x)|\le\delta_2$ for all $x\in\overline\Omega$. Thus, recalling ${\bf H}_g$ \ref{hg3} \ref{hg5}, ${\bf H}_f$ \ref{hf6}, we have
\begin{align*}
&\hat\varphi^+_\lambda(tu) = \int_\Omega H(\nabla (tu))\,dx+\frac{t^p}{p}\int_\Omega\xi(x)u^p\,dx+\frac{t^p}{p}\int_{\partial\Omega}\beta(x)u^p\,d\sigma-\int_\Omega\big[\lambda G(x,tu)+F(x,tu)\big]\,dx \\
&\le c_{42}t^r\|\nabla u\|_r^r+c_{43}\frac{t^p}{p}\|u\|^p-\lambda c_{44}\frac{t^q}{q}\|u\|^q \ (c_{43},c_{44}>0).
\end{align*}
Recalling that $q<r\le p$, we deduce that $\hat\varphi^+_\lambda(tu)<0$ for all $t>0$ small enough. Let $\rho_\lambda>0$ be as in Lemma \ref{mpp}, then we have
\[\inf_{\|u\|\le\rho_\lambda}\hat\varphi^+_\lambda(u)<0.\]
By ${\bf H}_a$ \ref{ha4} and the compact embeddings $\w\hookrightarrow L^p(\Omega),\,L^p(\partial\Omega)$, $\hat\varphi^+_\lambda$ is sequentially weakly l.s.c.\ in $\w$. So we can find $v_+\in\w$ s.t.\ $\|v_+\|\le\rho_\lambda$ and
\[\hat\varphi^+_\lambda(v_+)=\inf_{\|u\|\le\rho_\lambda}\hat\varphi^+_\lambda(u)<0<\hat m^+_\lambda.\]
In particular $v_+\neq 0,u_+$ and $\|v_+\|<\rho_\lambda$. Therefore $v_+\in K(\hat\varphi^+_\lambda)$. As above we deduce that $v_+\in D_+$ and is a solution of \eqref{rp}.
\end{proof}

\noindent
The case of $\hat\varphi^-_\lambda$ is simpler:

\begin{lemma}\label{cn}
If ${\bf H}$ hold, then for all $\lambda>0$ $\hat\varphi^-_\lambda\in C^1(\w)$ is coercive and satisfies $(C)$.
\end{lemma}
\begin{proof}
Preliminarily we prove that uniformly for a.a.\ $x\in\Omega$
\beq\label{cn1}
\lim_{t\to -\infty}\Big(\frac{c_2}{p(p-1)}\hat\lambda_1|t|^p-K^-_\lambda(x,t)\Big)=+\infty.
\eeq
Indeed, by ${\bf H}_e$ \ref{he1}, for all $\eta>0$ we can find $c_{45}=c_{45}(\eta)>0$ s.t.\ for a.a.\ $x\in\Omega$, all $t\le -c_{45}$
\[\lambda g(x,t)t+f(x,t)t-p\big(\lambda G(x,t)+F(x,t)\big)\ge\eta.\]
Recalling the definition of $K^-_\lambda$, for a.a.\ $x\in\Omega$, all $t<0$ we have
\begin{align*}
&\frac{d}{dt}\frac{K^-_\lambda(x,t)}{|t|^p} = \frac{\big(\lambda g(x,t)+f(x,t)+\mu |t|^{p-2}t\big)|t|^p-\big(\lambda G(x,t)+F(x,t)+\mu/p|t|^p\big)p|t|^{p-2}t}{|t|^{2p}} \\
&= \frac{\lambda g(x,t)t+f(x,t)t-p\big(\lambda G(x,t)+F(x,t)\big)}{|t|^pt} \le \frac{\eta}{|t|^pt}
\end{align*}
(recall that $t<0$). So, for a.a.\ $x\in\Omega$ and all $t'<t\le -c_{45}$ we have
\beq\label{cn2}
\frac{K^-_\lambda(x,t)}{|t|^p}-\frac{K^-_\lambda(x,t')}{|t'|^p} \le \int_{t'}^t\frac{\eta}{|\tau|^p\tau}\,d\tau \le \frac{\eta}{p}\Big(\frac{1}{|t'|^p}-\frac{1}{|t|^p}\Big).
\eeq
Besides, by ${\bf H}_g$ \ref{hg2}, ${\bf H}_f$ \ref{hf5} we have uniformly for a.a.\ $x\in\Omega$
\beq\label{cn3}
-c_{46} \le \liminf_{t\to -\infty}\frac{K^-_\lambda(x,t)}{|t|^p} \le \limsup_{t\to -\infty}\frac{K^-_\lambda(x,t)}{|t|^p} \le \frac{c_2 \hat\lambda_1}{p(p-1)}+\frac{\mu}{p} \ (c_{46}>0).
\eeq
Passing to the limit in \eqref{cn2} as $t'\to -\infty$ and applying \eqref{cn3}, we see that for a.a.\ $x\in\Omega$, all $t\le -c_{45}$
\[\frac{c_2 \hat\lambda_1}{p(p-1)}|t|^p-K^-_\lambda(x,t)\le\frac{\mu}{p}|t|^p-\frac{\eta}{p},\]
and the latter tends to $+\infty$ as $t\to -\infty$, yielding \eqref{cn1}.
\vskip2pt
\noindent
Now we prove coercivity, arguing by contradiction. Let $(u_n)$ be a sequence in $\w$ s.t.\ $\|u_n\|\to +\infty$ and $|\hat\varphi^-_\lambda(u_n)|\le c_{47}$ ($c_{47}>0$). For all $n\in\N$ set $w_n=\|u_n\|^{-1}u_n$, so $\|w_n\|=1$.  Passing if necessary to a subsequence, we have $w_n\rightharpoonup w$ in $\w$ and $w_n\to w$ in $L^p(\Omega)$ and $L^p(\partial\Omega)$. By Lemma \ref{ah} \ref{ah4} we have for all $n\in\N$
\beq\label{cn4}
\frac{c_2}{p(p-1)}\|\nabla w_n\|_p^p+\frac{1}{p}\int_\Omega(\xi(x)+\mu)|w_n|^p\,dx+\frac{1}{p}\int_{\partial\Omega}\beta(x)|w_n|^p\,d\sigma-\int_\Omega\frac{K^-_\lambda(x,u_n)}{\|u_n\|^p}\,dx\le\frac{c_{47}}{\|u_n\|^p}.
\eeq
By ${\bf H}_g$ \ref{hg1} \ref{hg2}, ${\bf H}_f$ \ref{hf1} \ref{hf5} we have for a.a.\ $x\in\Omega$, all $t\in\R$
\[|K^-_\lambda(x,t)|\le c_{48}(1+|t|^p) \ (c_{48}>0).\]
Thus, the sequence $(K^-_\lambda(\cdot,u_n)\|u_n\|^{-p})$ is uniformly integrable in $\Omega$. By the Dunford-Pettis theorem, then, it admits a weakly convergent subsequence in $L^1(\Omega)$. More precisely, using \eqref{cn3} and arguing as in \cite[Proposition 30]{APS}, we can find $\theta^-_\lambda\in L^\infty(\Omega)$ s.t.\ $\theta^-_\lambda(x)\le\hat\lambda_1$ for a.a.\ $x\in\Omega$ and
\[\frac{K^-_\lambda(\cdot,u_n)}{\|u_n\|^p} \rightharpoonup \Big(\frac{c_2}{p(p-1)}\theta^-_\lambda+\frac{\mu}{p}\Big)(w^-)^p \ \text{in $L^1(\Omega)$.}\]
Passing to the limit in \eqref{cn4} as $n\to\infty$, and recalling that by convexity
\[\|\nabla w\|_p^p\le\liminf_n\|\nabla w_n\|_p^p,\]
we get
\[\frac{c_2}{p(p-1)}\|\nabla w\|_p^p+\frac{1}{p}\int_\Omega(\xi(x)+\mu)|w|^p\,dx+\frac{1}{p}\int_{\partial\Omega}\beta(x)|w|^p\,d\sigma\le\int_\Omega\Big(\frac{c_2}{p(p-1)}\theta^-_\lambda(x)+\frac{\mu}{p}\Big)(w^-)^p\,dx.\]
Operating on both sides and recalling the definitions of $\xi_0$, $\beta_0$, we have
\beq\label{cn5}
\|\nabla w^-\|_p^p+\int_\Omega\xi_0(x)(w^-)^p\,dx+\int_{\partial\Omega}\beta_0(x)(w^-)^p\,d\sigma \le \int_\Omega\theta^-_\lambda(x)(w^-)^p\,dx.
\eeq
Now we distinguish two cases:
\begin{itemize}[leftmargin=1cm]
\item[$(a)$] If $\theta^-_\lambda\not\equiv\hat\lambda_1$ (non-resonance), then we can find $c_{49}>0$ s.t.\ for all $v\in\w$
\[\|\nabla v\|_p^p+\int_\Omega\xi_0(x)|v|^p\,dx+\int_{\partial\Omega}\beta_0(x)|v|^p\,d\sigma-\int_\Omega\theta^-_\lambda(x)|v|^p\,dx \ge c_{49}\|v\|^p.\]
So, by \eqref{cn5} we have $w^-=0$, i.e., $w\in\w_+$. Passing to the limit in \eqref{cn4}, we get
\[\frac{c_2}{p(p-1)}\|\nabla w\|_p^p+\frac{1}{p}\int_\Omega(\xi(x)+\mu)w^p\,dx+\frac{1}{p}\int_{\partial\Omega}\beta(x)w^p\,d\sigma\le 0,\]
hence $w=0$. Thus, using again \eqref{cn4} we see that $\|\nabla w_n\|_p\to 0$, which along with $w_n\to 0$ in $L^p(\Omega)$ yields $w_n\to 0$ in $\w$, against $\|w_n\|=1$.
\item[$(b)$] If $\theta^-_\lambda\equiv\hat\lambda_1$ (resonance), then \eqref{cn5} becomes
\[\|\nabla w^-\|_p^p+\int_\Omega\xi_0(x)(w^-)^p\,dx+\int_{\partial\Omega}\beta_0(x)(w^-)^p\,d\sigma\le\hat\lambda_1\|w^-\|_p^p.\]
By the Lagrange multiplier rule and the definition of $\hat\lambda_1$, either $w^-=0$, or $w^-\in\w_+$ is a principal eigenfunction of \eqref{ep}. If $w^-=0$, arguing as in case $(a)$ we reach a contradiction. So, let $w^-\in\w_+\setminus\{0\}$ be a principal eigenfunction of \eqref{ep}. Then $w^-\in D_+$, which implies $u_n(x)\to -\infty$ for a.a.\ $x\in\Omega$. Using \eqref{cn1} and Fatou's lemma, we have
\[\lim_n\int_\Omega\Big(\frac{c_2}{p(p-1)}\hat\lambda_1|u_n|^p-K^-_\lambda(x,u_n)\Big)\,dx=+\infty.\]
So, for all $n\in\N$ big enough
\begin{align*}
&\hat\varphi^-_\lambda(u_n) \ge \frac{c_2}{p(p-1)}\Big(\|\nabla u_n\|_p^p+\int_\Omega\xi_0(x)|u_n|^p\,dx+\int_{\partial\Omega}\beta_0(x)|u_n|^p\,d\sigma-\hat\lambda_1\|u_n\|_p^p\Big) \\
&+\int_\Omega\Big(\frac{c_2}{p(p-1)}\hat\lambda_1|u_n|^p-K^-_\lambda(x,u_n)\Big)\,dx,
\end{align*}
and the latter tends to $+\infty$ as $n\to\infty$, against $|\hat\varphi^-_\lambda(u_n)|\le c_{47}$.
\end{itemize}
In both cases we reach a contradiction, which proves that
\beq\label{cn6}
\lim_{\|u\|\to +\infty}\hat\varphi^-_\lambda(u)=+\infty.
\eeq
Now we prove that $\hat\varphi^-_\lambda$ satisfies $(C)$. Let $(u'_n)$ be a sequence in $\w$ s.t.\ $(\hat\varphi^-_\lambda(u'_n))$ is bounded in $\R$ and $(1+\|u'_n\|)(\hat\varphi^-_\lambda)'(u'_n)\to 0$ in $\w^*$. So, we can find a sequence $(\eps_n)$ s.t.\ $\eps_n\to 0^+$ and for all $n\in\N$, $v\in\w$
\[\Big|\langle A(u'_n),v\rangle+\int_\Omega(\xi(x)+\mu)|u'_n|^{p-2}u'_nv\,dx+\int_{\partial\Omega}\beta(x)|u'_n|^{p-2}u'_nv\,d\sigma-\int_\Omega k^-_\lambda(x,u'_n)v\,dx\Big|\le\frac{\eps\|v\|}{1+\|u'_n\|}.\]
By \eqref{cn6} $(u'_n)$ is bounded, so passing to a subsequence we have $u'_n\rightharpoonup u'$ in $\w$ and $u'_n\to u'$ in $L^p(\Omega)$ and $L^p(\partial\Omega)$. So, choosing $v=u'_n-u'$, we easily get
\[\limsup_n\langle A(u'_n),u'_n-u'\rangle \le 0,\]
which by the $(S)_+$-property of $A$ implies $u'_n\to u'$ in $\w$, concluding the proof.
\end{proof}

\noindent
By applying the direct method of the calculus of variations, we produce a negative solution:

\begin{proposition}\label{sn}
If ${\bf H}$ hold, then for all $\lambda>0$ problem \eqref{rp} admits at least one negative solution $u_-\in -D_+$.
\end{proposition}
\begin{proof}
Fix $\lambda>0$. By Lemma \ref{cn} $\hat\varphi^-_\lambda$ is coercive. Besides, it is sequentially weakly l.s.c., hence we can find $u_-\in\w$ s.t.
\[\hat\varphi^-_\lambda(u_-)=\inf_{u\in\w}\hat\varphi^-_\lambda(u).\]
Reasoning as in Proposition \ref{sp} we see that $\hat\varphi^-_\lambda(u_-)<0$, in particular $u_-\neq 0$. For all $v\in\w$ we have
\beq\label{sn1}
\langle A(u_-),v\rangle+\int_\Omega(\xi(x)+\mu)|u_-|^{p-2}u_-v\,dx+\int_{\partial\Omega}\beta(x)|u_-|^{p-2}u_-v\,d\sigma=\int_\Omega k^-_\lambda(x,u_-)v\,dx.
\eeq
Choosing $v=u_-^+$ in \eqref{sn1} and applying ${\bf H}_\beta$ and Lemma \ref{ah} \ref{ah3}, we get
\[\frac{c_2}{p-1}\|\nabla u_-^+\|_p^p+\int_\Omega(\xi(x)+\mu)(u_-^+)^p\,dx\le 0,\]
hence $u_-\in -\w_+\setminus\{0\}$. Then \eqref{sn1} becomes for all $v\in\w$
\[\langle A(u_-),v\rangle+\int_\Omega\xi(x)|u_-|^{p-2}u_-v\,dx+\int_{\partial\Omega}\beta(x)|u_-|^{p-2}u_-v\,d\sigma=\int_\Omega[\lambda g(x,u_-)+f(x,u_-)]v\,dx,\]
i.e., $u_-$ is a solution of \eqref{rp}. Reasoning as in \cite{PR4,PR5} we deduce $u\in L^\infty(\Omega)$. Then nonlinear regularity theory \cite{L} applies, yielding $u_-\in C^1(\overline\Omega)\setminus\{0\}$. By ${\bf H}_g$ \ref{hg3} we have for a.a.\ $x\in\Omega$, all $t\le 0$
\[g(x,t)\le 0,\]
while by ${\bf H}_f$ \ref{hf5} \ref{hf6} we can find $c_{50}>0$ s.t.\ for a.a.\ $x\in\Omega$, all $t\le 0$
\[f(x,t)t\ge -c_{50}|t|^p.\]
So, for a.a.\ $x\in\Omega$ we have 
\[{\rm div}\,a(\nabla (-u_-))\le(\|\xi\|_\infty+c_{50})(-u_-)^{p-1}.\]
By the nonlinear maximum principle \cite{PS} we get $-u_-\in D_+$, i.e., $u_-\in -D_+$.
\end{proof}

\section{Extremal constant sign solutions and nodal solution}\label{sec4}

\noindent
In this section our purpose is twofold: first we improve the results of Propositions \ref{sp} and \ref{sn} by proving that problem \eqref{rp} admits {\em extremal} constant sign solutions, i.e., a smallest positive solution and a biggest negative solution. Then we use truncations and a Morse-theoretic argument to prove existence of a {\em nodal} solution, thus completing the proof of Theorem \ref{main}.
\vskip2pt
\noindent
Preliminarily, we note that by ${\bf H}_g$ \ref{hg3}, ${\bf H}_f$ \ref{hf2} \ref{hf5} we can find $c_{51}>\|\xi\|_\infty$ s.t.\ for all $\lambda>0$, a.a.\ $x\in\Omega$, and all $t\in\R$
\beq\label{auxg}
\lambda g(x,t)t+f(x,t)t\ge\lambda c_9|t|^q-c_{51}(|t|^{p^*}+|t|^p).
\eeq
We introduce an auxiliary Robin problem (with critical growth):
\beq\label{auxp}
\begin{cases}
-{\rm div}\,a(\nabla u)+\xi(x)|u|^{p-2}u=\lambda c_9|u|^{q-2}u-c_{51}(|u|^{p^*-2}u+|u|^{p-2}u) & \text{in $\Omega$} \\
\displaystyle\frac{\partial u}{\partial n_a}+\beta(x)|u|^{p-2}u=0 & \text{on $\partial\Omega$.}
\end{cases}
\eeq
We prove an existence/uniqueness result for constant sign solutions of \eqref{auxp}:

\begin{proposition}\label{aux}
If ${\bf H}_a$, ${\bf H}_\xi$, ${\bf H}_\beta$ hold, then for all $\lambda>0$ \eqref{auxp} admits a unique positive solution $u_*\in D_+$ and a unique negative solution $v_*\in -D_+$.
\end{proposition}
\begin{proof}
First we prove existence of a positive solution. Set for all $u\in\w$
\begin{align*}
&\gamma^+_\lambda(u)=\int_\Omega H(\nabla u)\,dx+\frac{1}{p}\int_\Omega\xi(x)|u|^p\,dx+\frac{1}{p}\int_{\partial\Omega}\beta(x)|u|^p\,d\sigma+\frac{\mu}{p}\|u^-\|_p^p \\
&-\frac{\lambda c_9}{q}\|u^+\|_q^q+\frac{c_{51}}{p^*}\|u^+\|_{p^*}^{p^*}+\frac{c_{51}}{p}\|u^+\|_p^p,
\end{align*}
where $\mu>\|\xi\|_\infty$ is defined as in \eqref{mu}. By Lemma \ref{ah} \ref{ah4} and recalling that $c_{51}>\|\xi\|_\infty$, we have for all $u\in\w$
\begin{align*}
&\gamma^+_\lambda(u) \ge \frac{c_2}{p(p-1)}\|\nabla u\|_p^p+\frac{1}{p}\int_\Omega(\xi(x)+\mu)(u^-)^p\,dx \\
&+\frac{1}{p}\int_\Omega(\xi(x)+c_{51})(u^+)^p\,dx+\frac{c_{51}}{p^*}\|u^+\|_{p^*}^{p^*}-\frac{\lambda c_9}{q}\|u^+\|_q^q \\
&\ge c_{52}\|u\|^p-\lambda c_{53}\|u\|^q \ (c_{52},c_{53}>0),
\end{align*}
and the latter tends to $+\infty$ as $\|u\|\to +\infty$ (since $q<p$). Thus, $\gamma^+_\lambda$ is coercive in $\w$. Besides, $\gamma^+_\lambda$ is sequentially weakly l.s.c. So we can find $u_*\in\w$ s.t.
\[\gamma^+_\lambda(u_*)=\inf_{u\in\w}\gamma^+_\lambda(u).\]
Arguing as in Proposition \ref{sp} we see that $\gamma^+_\lambda(u_*)<0$, hence $u_*\neq 0$. For all $v\in\w$ we have
\begin{align*}
&\langle A(u_*),v\rangle+\int_\Omega\xi(x)|u_*|^{p-2}u_*v\,dx+\int_{\partial\Omega}\beta(x)|u_*|^{p-2}u_*v\,d\sigma+\mu\int_\Omega(u^-_*)^{p-1}v\,dx \\
&= \lambda c_9\int_\Omega(u^+_*)^{q-1}v\,dx-c_{51}\int_\Omega[(u^+_*)^{p^*-1}+(u^+_*)^{p-1}]v\,dx.
\end{align*}
Choosing $v=u^-_*$ yields, by Lemma \ref{ah} \ref{ah3},
\[\frac{c_2}{p-1}\|\nabla u^-_*\|_p^p+\int_\Omega(\xi(x)+\mu)(u^-_*)^p\,dx \le 0,\]
hence $u_*\in\w_+\setminus\{0\}$. As in previous cases we deduce that $u_*\in D_+$ and is a solution of \eqref{auxp}.
\vskip2pt
\noindent
Now we prove uniqueness. Set for all $u\in L^1(\Omega)$
\[\chi(u)=\begin{cases}
\displaystyle\int_\Omega H(\nabla(u^\frac{1}{r}))\,dx+\frac{1}{p}\int_\Omega(\xi(x)+c_{51})u^\frac{p}{r}\,dx+\frac{1}{p}\int_{\partial\Omega}\beta(x)u^\frac{p}{r}\,d\sigma & \text{if $u\ge 0$, $u^\frac{1}{r}\in\w$} \\
+\infty & \text{otherwise.}
\end{cases}\]
We claim that $\chi:L^1(\Omega)\to\R\cup\{+\infty\}$ is convex. Choose $u_1,u_2\in {\rm dom}\,(\chi)$, $\tau\in[0,1]$, and set for all $x\in\Omega$
\[u(x)=\big((1-\tau)u_1(x)+\tau u_2(x)\big)^\frac{1}{r},\]
then by \cite[Lemma 1]{DS} we have for a.a.\ $x\in\Omega$
\[|\nabla u(x)|\le\big((1-\tau)|\nabla(u_1^\frac{1}{r})(x)|^r+\tau|\nabla(u_2^\frac{1}{r})(x)|^r\big)^\frac{1}{r}.\]
By ${\bf H}_a$ \ref{ha1} \ref{ha4} we know that $H_0$ is increasing and $t\mapsto H_0(t^\frac{1}{r})$ is convex in $\R_+$. So we have for a.a.\ $x\in\Omega$
\begin{align*}
&H\big(\nabla((1-\tau)u_1+\tau u_2)^\frac{1}{r}(x)\big) = H_0(|\nabla u(x)|) \le H_0\Big(\big((1-\tau)|\nabla(u_1^\frac{1}{r})(x)|^r+\tau|\nabla(u_2^\frac{1}{r})(x)|^r\big)^\frac{1}{r}\Big) \\
&\le (1-\tau)H_0(|\nabla(u_1^\frac{1}{r})(x)|)+\tau H_0(|\nabla(u_2^\frac{1}{r})(x)|) = (1-\tau) H(\nabla(u_1^\frac{1}{r})(x))+\tau H(\nabla(u_2^\frac{1}{r})(x)),
\end{align*}
so the functional
\[u\mapsto \int_\Omega H(\nabla(u^\frac{1}{r}))\,dx\]
is convex. Besides, since $r\le p$, $c_{51}>\|\xi\|_\infty$, and $\beta\ge 0$, the functional
\[u\mapsto \frac{1}{p}\int_\Omega(\xi(x)+c_{51})u^\frac{p}{r}\,dx+\frac{1}{p}\int_{\partial\Omega}\beta(x)u^\frac{p}{r}\,d\sigma\]
is convex as well. Summarizing, we get the claim. By Fatou's lemma, $\chi$ is l.s.c.\ in $L^1(\Omega)$ and G\^ateaux differentiable at $C^1(\overline\Omega)$-functions.
\vskip2pt
\noindent
Now we assume that $u\in \w_+\setminus\{0\}$ is a solution of \eqref{auxp}. As usual we get $u\in D_+$, in particular $u\in{\rm int}\,(C^1(\overline\Omega)_+)$. So, for all $v\in C^1(\overline\Omega)$ and $t>0$ small enough, we have $u^r+tv\in{\rm int}\,(C^1(\overline\Omega)_+)$. Taking $t>0$ even smaller if necessary, we have as well $u_*^r+tv\in{\rm int}\,(C^1(\overline\Omega)_+)$. By what observed above, we have
\[\langle\chi'(u^r),v\rangle =\frac{1}{r}\int_\Omega\frac{-{\rm div}\,a(\nabla u)+(\xi(x)+c_{51})u^{p-1}}{u^{r-1}}\,v\,dx+\frac{1}{r}\int_\Omega\frac{\beta(x)u^{p-1}}{u^{r-1}}\,v\,d\sigma,\]
and a similar relation holds for $u_*$. Since $\chi'$ is a monotone operator, and recalling that $u$, $u_*$ solve \eqref{auxp}, we have
\begin{align*}
&0 \le \langle\chi'(u^r)-\chi'(u^r_*),u^r-u_*^r \rangle \\
&\le \frac{1}{r}\int_\Omega\frac{\lambda c_9u^{q-1}-c_{51}u^{p^*-1}}{u^{r-1}}(u^r-u_*^r)\,dx-\frac{1}{r}\int_\Omega\frac{\lambda c_9u_*^{q-1}-c_{51}u_*^{p^*-1}}{u_*^{r-1}}(u^r-u_*^r)\,dx \\
&= \frac{1}{r}\int_\Omega\big[\lambda c_9(u^{q-r}-u_*^{q-r})-c_{51}(u^{p^*-r}-u_*^{p^*-r})\big](u^r-u_*^r)\,dx \le 0
\end{align*}
(recall that $q<r<p^*$). So we have $u=u_*$, i.e., $u_*$ is the only positive solution of \eqref{auxp}.
\vskip2pt
\noindent
Since problem \eqref{auxp} is odd, it clearly has a unique negative solution $v_*=-u_*\in -D_+$, which concludes the proof.
\end{proof}

\noindent
From now on, for all $u,v\in\w$ we shall write $u\le v$ meaning that $u(x)\le v(x)$ for a.a.\ $x\in\Omega$. Such partial ordering makes $\w$ an ordered Banach space, and for all set $S\subset\w$ we will use accordingly the notions of minorant, majorant, infimum, and supremum of $S$. Similarly, if $u_1\le u_2$ we set
\[[u_1,u_2]=\{u\in\w:\,u_1\le u\le u_2\}.\]
Now we go back to problem \eqref{rp}. For all $\lambda>0$ we denote by $S(\lambda)$ (resp.\ $S_+(\lambda)$, $S_-(\lambda)$) the set of all solutions (resp.\ positive, negative solutions) of \eqref{rp}. From Proposition \ref{sp} we know that $\emptyset\neq S_+(\lambda)\subseteq D_+$ for all $\lambda\in(0,\lambda^*)$, while Proposition \ref{sn} tells us that $\emptyset\neq S_-(\lambda)\subseteq -D_+$ for all $\lambda>0$. Moreover, from \cite{PR6}, \cite{FP} we know that for all $\lambda>0$ the set $S_+(\lambda)$ is downward directed, i.e., for all $u_1,u_2\in S_+(\lambda)$ we can find $u_3\in S_+(\lambda)$ s.t.\ $u_3\le u_1$ and $u_3\le u_2$. Similarly, $S_-(\lambda)$ is upward directed. Now we prove a lower bound for $S_+(\lambda)$ and an upper bound for $S_-(\lambda)$, respectively:

\begin{lemma}\label{bd}
If ${\bf H}$ hold, then
\begin{enumroman}
\item\label{bdp} for all $\lambda\in(0,\lambda^*)$ and all $u\in S_+(\lambda)$, $u\ge u_*$;
\item\label{bdn} for all $\lambda>0$ and all $u\in S_-(\lambda)$, $u\le v_*$.
\end{enumroman}
\end{lemma}
\begin{proof}
We prove \ref{bdp}. Fix $\lambda\in(0,\lambda^*)$, $u\in S_+(\lambda)$. Since $u\in D_+$, we can set for all $(x,t)\in\Omega\times\R$
\[h^+_\lambda(x,t)=\begin{cases}
0 & \text{if $t<0$} \\
\lambda c_9t^{q-1}-c_{51}t^{p^*-1}-(c_{51}-\mu)t^{p-1} & \text{if $0\le t\le u(x)$} \\
\lambda c_9u(x)^{q-1}-c_{51}u(x)^{p^*-1}-(c_{51}-\mu)u(x)^{p-1} & \text{if $t>u(x)$}
\end{cases}\]
and
\[H^+_\lambda(x,t)=\int_0^t h^+_\lambda(x,\tau)\,d\tau.\]
Then we define a functional $\hat\gamma^+_\lambda\in C^1(\w)$ by setting for all $u\in\w$
\[\hat\gamma^+_\lambda(u)=\int_\Omega H(\nabla u)\,dx+\frac{1}{p}\int_\Omega(\xi(x)+\mu)|u|^p\,dx+\frac{1}{p}\int_{\partial\Omega}\beta(x)|u|^p\,d\sigma-\int_\Omega H^+_\lambda(x,u)\,dx.\]
Since $\mu>\|\xi\|_\infty$, and $h^+_\lambda(x,\cdot)$ is bounded in $\R$ for a.a.\ $x\in\R$, $\hat\gamma^+_\lambda$ is coercive and sequentially weakly l.s.c. So there exists $\hat u_*\in\w$ s.t.
\[\hat\gamma^+_\lambda(\hat u_*)=\inf_{u\in\w}\hat\gamma^+_\lambda(u)<0,\]
in particular $\hat u_*\neq 0$. For all $v\in\w$ we have
\beq\label{bd1}
\langle A(\hat u_*),v\rangle+\int_\Omega(\xi(x)+\mu)|\hat u_*|^{p-2}\hat u_*v\,dx+\int_{\partial\Omega}\beta(x)|\hat u_*|^{p-2}\hat u_*v\,d\sigma=\int_\Omega h^+_\lambda(x,\hat u_*)v\,dx.
\eeq
Choosing $v=-\hat u^-_*$ in \eqref{bd1}, we have by Lemma \ref{ah} \ref{ah3}
\[\frac{c_2}{p-1}\|\nabla \hat u_*^-\|_p^p+\int_\Omega(\xi(x)+\mu)(\hat u_*^-)^p\,dx\le 0,\]
hence $\hat u_*\in\w_+\setminus\{0\}$. Instead, choosing $v=(\hat u_*-u)^+$ in \eqref{bd1}, applying \eqref{auxg}, and recalling that $u\in S_+(\lambda)$ yields
\begin{align*}
&\langle A(\hat u_*),(\hat u_*-u)^+\rangle+\int_\Omega(\xi(x)+\mu)\hat u_*^{p-1}(\hat u_*-u)^+\,dx+\int_{\partial\Omega}\beta(x)\hat u_*^{p-1}(\hat u_*-u)^+\,d\sigma \\
&= \int_\Omega[\lambda c_9u^{q-1}-c_{51}u^{p^*-1}-(c_{51}-\mu)u^{p-1}](\hat u_*-u)^+\,dx \\
&\le \int_\Omega[\lambda g(x,u)+f(x,u)+\mu u^{p-1}](\hat u_*-u)^+\,dx \\
&= \langle A(u),(\hat u_*-u)^+\rangle+\int_\Omega(\xi(x)+\mu)u^{p-1}(\hat u_*-u)^+\,dx+\int_{\partial\Omega}\beta(x)u^{p-1}(\hat u_*-u)^+\,d\sigma,
\end{align*}
hence
\[\langle A(\hat u_*)-A(u),(\hat u_*-u)^+\rangle+\int_\Omega(\xi(x)+\mu)(\hat u_*^{p-1}-u^{p-1})(\hat u_*-u)^+\,dx\le 0.\]
By Lemma \ref{ah} \ref{ah1}, this implies $\hat u_*\le u$. So \eqref{bd1} becomes for all $v\in\w$
\[\langle A(\hat u_*),v\rangle+\int_\Omega\xi(x)\hat u_*^{p-1}v\,dx+\int_{\partial\Omega}\beta(x)\hat u_*^{p-1}v\,d\sigma=\int_\Omega [\lambda c_9\hat u_*^{q-1}-c_{51}(\hat u_*^{p^*-1}+\hat u_*^{p-1}]v\,dx.\]
We conclude that $\hat u_*$ is a positive solution of \eqref{auxp}, hence by Proposition \ref{aux} we have $\hat u_*=u_*$. Thus, we have $u\ge u_*$.
\vskip2pt
\noindent
Similarly we prove \ref{bdn}.
\end{proof}

\noindent
Using these bounds, we can detect extremal constant sign solutions of \eqref{rp}:

\begin{proposition}\label{ex}
If ${\bf H}$ hold, then
\begin{enumroman}
\item\label{exp} for all $\lambda\in(0,\lambda^*)$ there exists $\underline u_+\in S_+(\lambda)$ s.t.\ $\underline u_+=\inf\,S_+(\lambda)$;
\item\label{exn} for all $\lambda>0$ there exists $\overline u_-\in S_-(\lambda)$ s.t.\ $\overline u_-=\sup\,S_-(\lambda)$.
\end{enumroman}
\end{proposition}
\begin{proof}
We prove \ref{exp}. 
By \cite[Lemma 3.10, p.\ 178]{HP} we can find a sequence $(u_n)$ in $S_+(\lambda)$, pointwise decreasing, s.t.
\beq\label{ex0}
\inf_{n\in\N}\,u_n=\inf\,S_+(\lambda).
\eeq
For all $n\in\N$, $v\in\w$ we have
\beq\label{ex1}
\langle A(u_n),v\rangle+\int_\Omega\xi(x)u_n^{p-1}v\,dx+\int_{\partial\Omega}\beta(x)u_n^{p-1}v\,d\sigma=\int_\Omega[\lambda g(x,u_n)+f(x,u_n)]v\,dx.
\eeq
Choosing $v=u_n$ in \eqref{ex1}, recalling that $0\le u_n\le u_1$, and using Lemma \ref{ah} \ref{ah3}, we see that $(u_n)$ is bounded in $\w$. Passing to a subsequence, we have $u_n\rightharpoonup\underline u_+$ in $\w$, $u_n\to \underline u_+$ in $L^p(\Omega)$ and $L^p(\partial\Omega)$. In particular $\underline u_+\in\w_+$. Choosing $v=u_n-\underline u_+$ in \eqref{ex1} and passing to the limit as $n\to\infty$ then provides
\[\lim_n\langle A(u_n),u_n-\underline u_+\rangle=0,\]
which by the $(S)_+$ property of $A$ implies $u_n\to\underline u_+$ in $\w$. Once again we use \eqref{ex1} and for all $v\in\w$ we have
\[\langle A(\underline u_+),v\rangle+\int_\Omega\xi(x)\underline u_+^{p-1}v\,dx+\int_{\partial\Omega}\beta(x)\underline u_+^{p-1}v\,d\sigma=\int_\Omega[\lambda g(x,\underline u_+)+f(x,\underline u_+)]v\,dx.\]
So $\underline u_+\in S(\lambda)$. Lemma \ref{bd} \ref{bdp} implies $\underline u_+\ge u_*>0$, so $\underline u_+\in S_+(\lambda)$. Then by \eqref{ex0} we have $u\ge\underline u_+$ for all $u\in S_+(\lambda)$.
\vskip2pt
\noindent
Similarly we prove \ref{exn}.
\end{proof}

\noindent
Let us recall a basic notion from Morse theory (see \cite[Definition 6.43]{MMP}). Let $\varphi$ be a $C^1$-functional defined on a Banach space $X$, and $u\in X$ be an isolated critical point of $\varphi$, i.e., there exists a neighborhood $U$ of $u$ s.t.\ $u$ is the only critical point of $\varphi$ in $U$. For all $k\in\N$, we define the $k$-{\em th critical group} of $\varphi$ at $u$ as
\[C_k(\varphi,u)=H_k\big(\{v\in U:\,\varphi(v)\le\varphi(u)\},\{v\in U:\,\varphi(v)\le\varphi(u),\,v\neq u\}\big),\]
where $H_k(\cdot,\cdot)$ denotes the $k$-th singular homology group of a topological pair of sets (such definition is independent of $U$). Now we can perform our final step and produce a nodal solution:

\begin{proposition}\label{nod}
If ${\bf H}$ hold, then for all $\lambda\in(0,\lambda^*)$ problem \eqref{rp} admits a nodal solution $\tilde u\in C^1(\overline\Omega)\setminus\{0\}$ s.t.\ for all $x\in\Omega$
\[\overline u_-(x)\le\tilde u(x)\le\underline u_+(x).\]
\end{proposition}
\begin{proof}
Fix $\lambda\in(0,\lambda^*)$, and let $\overline u_-\in S_-(\lambda)$, $\underline u_+\in S_+(\lambda)$ be given by Proposition \ref{ex}. We set for all $(x,t)\in\Omega\times\R$
\[\tilde k_\lambda(x,t)=\begin{cases}
\lambda g(x,\overline u_-(x))+f(x,\overline u_-(x))+\mu|\overline u_-(x)|^{p-2}\overline u_-(x) & \text{if $t<\overline u_-(x)$} \\
\lambda g(x,t)+f(x,t)+\mu|t|^{p-2}t & \text{if $\overline u_-(x)\le t\le\underline u_+(x)$} \\
\lambda g(x,\underline u_+(x))+f(x,\underline u_+(x))+\mu\underline u_+(x)^{p-1} & \text{if $t>\underline u_+(x)$}
\end{cases}\]
(with $\mu>0$ given by \eqref{mu}), as well as
\[\tilde K_\lambda(x,t)=\int_0^t\tilde k_\lambda(x,\tau)\,d\tau.\]
Now we set for all $u\in\w$
\[\tilde\varphi_\lambda(u)=\int_\Omega H(\nabla u)\,dx+\frac{1}{p}\int_\Omega(\xi(x)+\mu)|u|^p\,dx+\frac{1}{p}\int_{\partial\Omega}\beta(x)|u|^p\,d\sigma-\int_\Omega\tilde K_\lambda(x,u)\,dx.\]
Clearly $\tilde\varphi_\lambda\in C^1(\w)$. We study now the properties of its critical set:
\beq\label{nod1}
K(\tilde\varphi_\lambda)\subseteq[\overline u_-,\underline u_+]\cap S(\lambda).
\eeq
Indeed, let $u\in K(\tilde\varphi_\lambda)$. For all $v\in\w$ we have
\beq\label{nod2}
\langle A(u),v\rangle+\int_\Omega(\xi(x)+\mu)|u|^{p-2}uv\,dx+\int_{\partial\Omega}\beta(x)|u|^{p-2}uv\,d\sigma=\int_\Omega\tilde k_\lambda(x,u)v\,dx.
\eeq
Choosing $v=(u-\underline u_+)^+$ in \eqref{nod2} we have
\begin{align*}
&\langle A(u),(u-\underline u_+)^+\rangle+\int_\Omega(\xi(x)+\mu)|u|^{p-2}u(u-\underline u_+)^+\,dx+\int_{\partial\Omega}\beta(x)|u|^{p-2}u(u-\underline u_+)^+\,d\sigma \\
&= \int_\Omega[\lambda g(x,\underline u_+)+f(x,\underline u_+)+\mu\underline u_+^{p-1}](u-\underline u_+)^+\,dx \\
&= \langle A(\underline u_+),(u-\underline u_+)^+\rangle+\int_\Omega(\xi(x)+\mu)\underline u_+^{p-1}(u-\underline u_+)^+\,dx+\int_{\partial\Omega}\beta(x)\underline u_+^{p-1}(u-\underline u_+)^+\,d\sigma,
\end{align*}
i.e.,
\[\langle A(u)-A(\underline u_+),(u-\underline u_+)^+\rangle+\int_\Omega(\xi(x)+\mu)(u^{p-1}-\underline u_+^{p-1})(u-\underline u_+)^+\,dx\le 0.\]
Since $\mu>\|\xi\|_\infty$, we have $u\le\underline u_+$. Similarly, choosing $v=(\overline u_--u)^+$ in \eqref{nod2} we get $u\ge\overline u_-$. Thus, \eqref{nod2} becomes for all $v\in\w$
\[\langle A(u),v\rangle+\int_\Omega\xi(x)|u|^{p-2}uv\,dx+\int_{\partial\Omega}\beta(x)|u|^{p-2}uv\,d\sigma=\int_\Omega [\lambda g(x,u)+f(x,u)]v\,dx,\]
hence $u\in S(\lambda)$. In particular, by nonlinear regularity theory \cite{L} we have $u\in C^1(\overline\Omega)$.
\vskip2pt
\noindent
Now set for all $(x,t)\in\Omega\times\R$
\[\tilde k^\pm_\lambda(x,t)=\tilde k_\lambda(x,\pm t^\pm), \ \tilde K^\pm_\lambda(x,t)=\int_0^t\tilde k^\pm_\lambda(x,\tau)\,d\tau,\]
and the corresponding functionals $\tilde\varphi^\pm_\lambda\in C^1(\w)$ defined for all $u\in\w$ by
\[\tilde\varphi^\pm_\lambda(u)=\int_\Omega H(\nabla u)\,dx+\frac{1}{p}\int_\Omega(\xi(x)+\mu)|u|^p\,dx+\frac{1}{p}\int_{\partial\Omega}\beta(x)|u|^p\,d\sigma-\int_\Omega\tilde K^\pm_\lambda(x,u)\,dx.\]
For the critical sets of such functionals we have a complete description:
\beq\label{nod3}
K(\tilde\varphi^+_\lambda)=\{0,\underline u_+\}.
\eeq
Indeed, clearly $0,\underline u_+\in K(\tilde\varphi^+_\lambda)$. Besides, reasoning as above we see that for all $u\in K(\tilde\varphi^+_\lambda)\setminus\{0\}$
\[u\in[0,\underline u_+]\cap S_+(\lambda).\]
Then, by Proposition \ref{ex} \ref{exp} we have $u\ge\underline u_+$, hence $u=\underline u_+$. Similarly we get
\beq\label{nod4}
K(\tilde\varphi^-_\lambda)=\{0,\overline u_-\}.
\eeq
We prove now that $\underline u_+$, $\overline u_-$ are local minimizers of $\tilde\varphi_\lambda$. We only deal with $\underline u_+$, as the case of $\overline u_-$ is analogous. Since $\tilde k^+_\lambda(x,\cdot)$ is bounded in $\R$ for a.a.\ $x\in\Omega$, the functional $\tilde\varphi^+_\lambda$ is coercive, beside being sequentially weakly l.s.c. So we can find $\tilde u_+\in\w$ s.t.
\[\tilde\varphi^+_\lambda(\tilde u_+)=\inf_{u\in\w}\tilde\varphi^+_\lambda(u)<0,\]
in particular $\tilde u_+\neq 0$. By \eqref{nod3}, then, we have $\tilde u_+=\underline u_+$, i.e., $\underline u_+$ is a global minimizer of $\tilde\varphi^+_\lambda$. We recall that
\[\underline u_+\in D_+\subseteq{\rm int}\,(C^1(\overline\Omega)_+),\]
and clearly the functionals $\tilde\varphi_\lambda$, $\tilde\varphi^+_\lambda$ agree on $C^1(\overline\Omega)_+$. So $\underline u_+$ is a $C^1(\overline\Omega)$-local minimizer of $\tilde\varphi_\lambda$, namely there exists $\rho>0$ s.t.\ for all $u\in C^1(\overline\Omega)$ with $\|u-\underline u_+\|_{C^1(\overline\Omega)}<\rho$ we have
\[\tilde\varphi_\lambda(u)\ge\tilde\varphi_\lambda(\underline u_+).\]
By the results of \cite{PR5}, $\underline u_+$ is as well a $\w$-local minimizer of $\tilde\varphi_\lambda$, namely there exists $c_{54}=c_{54}(\rho)>0$ s.t.\ for all $u\in\w$ with $\|u-\underline u_+\|<c_{54}$ we have
\[\tilde\varphi_\lambda(u)\ge\tilde\varphi_\lambda(\underline u_+).\]
Without loss of generality we may assume $\tilde\varphi_\lambda(\overline u_-)\le\tilde\varphi_\lambda(\underline u_+)$, and that the set $K(\tilde\varphi_\lambda)$ is finite. Clearly $\tilde\varphi_\lambda$ satisfies $(C)$. By \cite[Proposition 29]{APS}, then, we can find $\tilde\rho_\lambda\in(0,\|\overline u_--\underline u_+\|)$ s.t., summarizing,
\beq\label{nod5}
\tilde\varphi_\lambda(\overline u_-)\le\tilde\varphi_\lambda(\underline u_+)<\tilde m_\lambda=\inf_{\|u-\underline u_+\|=\tilde\rho_\lambda}\tilde\varphi_\lambda(u).
\eeq
Applying a convenient version of the mountain pass theorem \cite{H}, we see that there exists $\tilde u\in K(\tilde\varphi_\lambda)$ of mountain pass type, namely, s.t.\ the set
\[\{u\in\w:\,\|u-\tilde u\|\le\rho,\,\tilde\varphi_\lambda(u)<\tilde\varphi_\lambda(\tilde u)\}\]
is path disconnected for all $\rho>0$ small enough, and $\tilde\varphi_\lambda(\tilde u)\ge\tilde m_\lambda$. By \eqref{nod5} we have $\tilde u\neq\underline u_+,\overline u_-$, so it remains to prove that $\tilde u\neq 0$. Indeed, since $\tilde u$ is of mountain pass type, by \cite[Corollary 6.81]{MMP2} we have
\beq\label{nod6}
C_1(\tilde\varphi_\lambda,\tilde u)\neq 0.
\eeq
Let $\delta_0,\delta_1>0$ be as in ${\bf H}_g$ \ref{hg5}, ${\bf H}_f$ \ref{hf7}, respectively, and set
\[\delta_3=\min\Big\{\delta_0,\delta_1,\min_{\overline\Omega}\underline u_+,\min_{\overline\Omega}(-\overline u_-)\Big\}>0\]
(recall that $\underline u_+,-\overline u_-\in D_+$). Fix $c_{55}>q$, $\delta\in(0,\delta_3)$. Then for a.a.\ $x\in\Omega$, all $|t|\le\delta$ we have by ${\bf H}_g$ \ref{hg3} \ref{hg5}, ${\bf H}_f$ \ref{hf6} \ref{hf7}
\begin{align*}
&c_{55}\big[\lambda G(x,t)+F(x,t)\big]-[\lambda g(x,t)t+f(x,t)t] \\
&\ge \lambda\big[(c_{55}-q)G(x,t)+\big(qG(x,t)-g(x,t)t\big)\big]-f(x,t)t \\
&\ge \lambda\frac{(c_{55}-q)c_9}{q}|t|^q-c_{56}|t|^p \ (c_{56}>0),
\end{align*}
and the latter is positive for $\delta>0$ small enough (as $q<p$). Thus we can apply \cite[Proposition 6]{PR7} and get for all $k\in\N$
\beq\label{nod7}
C_k(\tilde\varphi_\lambda,0)=0.
\eeq
Comparing \eqref{nod6} and \eqref{nod7} we see that $\tilde u\neq 0$. Furthermore, by \eqref{nod1} we have $\overline u_-\le\tilde u\le\underline u_+$, so as above we get $\tilde u\in S(\lambda)$. Proposition \ref{ex} now implies that $\tilde u$ is nodal.
\end{proof}

\noindent
To conclude, we note that Theorem \ref{main} follows at once from Propositions \ref{sp}, \ref{sn}, and \ref{nod}.

\bigskip

\noindent
{\small {\bf Acknowledgement.} A.\ Iannizzotto and M.\ Marras are members of the Gruppo Nazionale per l'Analisi Matematica, la Probabilit\`a e le loro Applicazioni (GNAMPA) of the Istituto Nazionale di Alta Matematica (INdAM). This work was partially performed during a visit of N.S.\ Papageorgiou at the University of Cagliari, funded by the local Visiting Professor Programme.}

\end{document}